\documentclass[
final,
UKenglish,			
a4paper,			
bibliography=totoc	
]{scrartcl}

\author{Marcel Fenzl\footnote{University of Zurich, Winterthurerstrasse 190, 8057 Zürich, Switzerland. \newline Email: \href{mailto:marcel.fenzl@math.uzh.ch}{marcel.fenzl@math.uzh.ch}}}
\title{Asymptotic results for stabilizing functionals of point processes having fast decay of correlations}

\makeatletter
\def\blfootnote{\gdef\@thefnmark{}\@footnotetext}
\makeatother

\newcommand{\subjclass}[1]{\blfootnote{\textup{2010} \textit{Mathematics Subject Classification.} #1}}
\newcommand{\keywords}[1]{\blfootnote{\textit{Key words and phrases.} #1}}

\usepackage{babel}
\usepackage[T1]{fontenc}
\usepackage[utf8]{inputenc}
\usepackage[useregional]{datetime2} 

\usepackage{lmodern}
\usepackage{microtype}

\usepackage{fzlmath}

\usepackage{doi}
\usepackage{url}

\usepackage{enumitem}
\setlist[enumerate,1]{label={(\roman*)}}

\usepackage{hyperref}
\hypersetup{
	pdfauthor={Marcel Fenzl},
	pdftitle={Asymptotic results for stabilizing functionals of point processes having fast decay of correlations},
	pdfsubject={2010 Mathematics Subject Classification: Primary 60F10, 60D05; Secondary 60G55, 05C80, 52A22},
	pdfkeywords={Stabilizing functionals, point processes having fast decay of correlations, explicit bounds, cumulants, random graphs, random packing, determinantal point processes, Gibbs point processes, Berry-Esseen bounds, moderate deviations, concentration inequalities},
	bookmarksopen={true},
	final
}

\usepackage{fzlthm} 

\usepackage{xparse}

\ExplSyntaxOn
\NewDocumentCommand{\multiadjustlimits}{m}
 {
  \group_begin:
  \multiadjustlimits_measure:n { #1 }
  \multiadjustlimits_print:n { #1 }
  \group_end:
 }

\tl_new:N  \l__multiadjustlimits_operator_tl
\tl_new:N  \l__multiadjustlimits_limit_tl

\cs_new_protected:Nn \multiadjustlimits_measure:n
 {
  \clist_map_function:nN { #1 } \__multiadjustlimits_measure:n
 }
\cs_new_protected:Nn \__multiadjustlimits_measure:n
 {
  \__multiadjustlimits_measure:NNn #1
 }
\cs_new_protected:Nn \__multiadjustlimits_measure:NNn
 {
  \tl_put_right:Nn \l__multiadjustlimits_operator_tl { #1 }
  \tl_put_right:Nn \l__multiadjustlimits_limit_tl { #3 }
 }

\cs_new_protected:Nn \multiadjustlimits_print:n
 {
  \clist_map_function:nN { #1 } \__multiadjustlimits_print:n
 }
\cs_new_protected:Nn \__multiadjustlimits_print:n
 {
  \__multiadjustlimits_print:NNn #1
 }
\cs_new_protected:Nn \__multiadjustlimits_print:NNn
 {
  \mathop { \vphantom{\l__multiadjustlimits_operator_tl} \mathopen{} #1 }
  \limits
  \sb{ \vphantom{\l__multiadjustlimits_limit_tl} #3 }
 }

\ExplSyntaxOff

\DeclareMathOperator{\dist}{dist}

\usepackage{bm}
\newcommand{\mult}[1]{\bm{#1}}

\DeclareMathOperator{\EDC}{EDC}
\DeclareMathOperator{\BC}{BC}
\DeclareMathOperator{\ST}{S}
\DeclareMathOperator{\MG}{MG}
\DeclareMathOperator{\PG}{PG}
\DeclareMathOperator{\SIG}{SIG}

\begin{document}

\maketitle



\begin{abstract}
We establish precise bounds on cumulants for a rather general class of non-linear geometric functionals satisfying the stabilization property under a simple, stationary (marked) point process admitting fast decay of its correlation functions and thereby conclude a Berry-Esseen bound, a concentration inequality, a moderate deviation principle and a Marcinkiewicz-Zygmund-type strong law of large numbers. The result is applied to the germ-grain model as well as to random sequential absorption for $α$-determinantal point processes having fast decaying kernels and certain Gibbsian point processes. The proof relies on cumulant expansions using a clustering result as well as factorial moment expansions for point processes.
\end{abstract}




\section{Introduction and main results}

\subsection{General introduction}

Consider a stationary point process $\cP$ on $ℝ^d$ and its restriction $\cP_n = \cP\cap W_n$ to the box $W_n = \intcc{-\frac{1}{2}n^\frac{1}{d},\frac{1}{2}n^\frac{1}{d}}^d$ of volume $n$. Global geometric statistics of such point processes can often be described in terms of local contributions, i.e.\ the geometric statistic can be decomposed as a sum of spatially dependent terms in the form
\begin{equation}\label{eq:GeomStat}
	\sum_{x∈\cP_n} ξ\bigl(x,\cP_n\bigr).
\end{equation}
Here, the so-called score function $ξ$ depends on a point $x∈\cP_n$ as well as on the whole point configuration $\cP_n$ and takes values in $ℝ$. In particular, such statistics can be non-linear. It is generally impossible to reckon upon asymptotic results like laws of large numbers, central limit theorems or deviation probabilities for such geometric statistics, but under suitable locality conditions on the score function $ξ$ together with some form of independence between the points of the point process $\cP$ an asymptotic treatment becomes possible. In this article, we provide explicit bounds on cumulants and thereby establish Berry-Esseen bounds, concentration inequalities, moderate deviation principles and Marcinkiewicz-Zygmund-type strong laws of large numbers for such geometric statistics. Statistics we are able to investigate arise from geometric structures as various as random graphs, germ-grain models as well as random sequential packing and its extensions to name a few (see \cref{sec:Examples} for more details).

The notion of the score depending on local data only can be made precise by using the concept of stabilization. Roughly speaking, it requires the (random) range of dependence of the score function $ξ$ at $x∈\cP$ to be small in an appropriate sense. For a precise definition of stabilization consult \cref{def:ST}. Many asymptotic results are available for stabilizing statistics under Poisson and binomial input by now. The concept of stabilization was established in a series of works by Penrose and Yukich (see \cite{PY01CLTComputationalGeometry,PY02LimitTheoryRandomSequentialPacking,PY03WLLNGeometricProbability,PY05NormalApproximationGeometricProbability}) in which they prove weak laws of large numbers, central limit theorems and Berry-Esseen estimates for various statistics of e.g.\ the $k$-nearest neighbour graph, the sphere of influence graph, Voronoi tessellation and random sequential packing under Poisson point process input and under binomial input. Afterwards, the idea of stabilizing score functions gained much popularity in the study of geometric statistics. It turned out that instead of studying \cref{eq:GeomStat} directly considering the $ξ$-weighted empirical measure
\begin{equation*}
	μ_n^ξ = \sum_{x∈\cP_n} ξ\bigl(x,\cP_n\bigr) δ_{xn^{-1/d}}
\end{equation*}
and its evaluation against test functions $f\colon ℝ^d \to ℝ$ given by
\begin{equation}\label{eq:GeomStatEmpTested}
	μ_n^ξ(f) = \sum_{x∈\cP_n} ξ\bigl(x,\cP_n\bigr) f\bigl(xn^{-\frac{1}{d}}\bigr)
\end{equation}
provides more insight. For $μ_n^ξ(f)$, \cite{P07LLNStochasticGeometry} provides a strong law of large numbers. Further central limit theorems were established: In \cite{BY05GaussianLimitsGeometricProbability,P07GaussianLimitsRandomGeometricMeasure} central limit theorems were proven in the context of the random measure $μ_n^ξ$ and in \cite{P05MultivariateSpatialCLT} functional central limit theorems were considered. Refinements on the speed of convergence were obtained in \cite{ET14BENonlinearFunctionals} by using a Malliavin Stein approach. With the same approach, in \cite{LPS16NormalApproximation} the authors could prove conjecturally optimal Berry-Esseen bounds for various functionals in the case of Poisson input. A large deviation principle for stabilizing functionals was established in \cite{SY05LDPSpatialPointProcesses}. Moderate deviation principles bridging between the scale of the central limit theorem and the large deviation principle were proven in \cite{BESY08MDPGeometricProbability,ES10ProcessLevelMDP} for some functionals and later in \cite{ERS15MDPForStabilizingFunctionals} in more generality for stabilizing functionals under Poisson input. A survey about the concept of stabilization is provided in \cite{Y13LimitTheoremsStochastikGeometry}.

While all the previous articles only treat the case of Poisson or binomial point process input, the results here also apply to input different than but sufficiently close to the independent Poissonian one. A fruitful concept to characterize a point process as being close to the Poisson point process is the concept of exponentially fast decay of correlation functions. Roughly speaking, correlation functions are called to decay fast if they factor up to an (exponentially) small error. A precise formulation of this concept can be found in \cref{def:EDC}. Leading examples falling into the class of point processes with exponentially fast decay of correlations are $α$-determinantal point processes with decaying kernel and certain Gibbsian point processes. A first result proving a central limit theorem for linear statistics, i.e.\ $ξ$ depending only on $x$ and not on $\cP$, of determinantal point processes is provided in \cite{S02GaussianLimitDPP}. This is further extended to linear statistics of $α$-determinantal point processes in \cite{ST03RandomPointFieldsI}. The idea of exploiting exponentially fast decay of correlation functions goes back to \cite{M75CLTGibbsianRandomFields}, in which the concept was applied to linear statistics of certain Gibbsian measures. Various limiting results for the above-mentioned non-linear statistics of Gibbsian point processes were established in \cite{SY13GeometricFunctionalsGibbs}. By applying the idea of exponentially fast decay of correlations,  Błaszczyszyn, Yogeshwaran and Yukich established a unified approach to non-linear geometric statistics under all these different point processes in \cite{BYY19GeometricStatistics}. They prove both laws of large numbers and central limit theorems.

Our work extends the results found in \cite{BYY19GeometricStatistics} by providing an explicit bound on cumulants for the geometric statistic $μ_n^ξ(f)$ from \cref{eq:GeomStatEmpTested} (see \cref{thm:BoundCumulants}). By doing so, we can apply general results from the Lithuanian school \cite{SS91LimitTheorems} to translate the bound on cumulants into asymptotic results and thus add Berry-Esseen bounds (\cref{thm:BerryEsseen}), concentration inequalities (\cref{thm:Concentration}), moderate deviation principles (\cref{thm:MDP}) and Marcinkiewicz-Zygmund-type strong laws of large numbers (\cref{thm:SLLN}) to the known central limit theorems. This answers an open question posed in \cite[Remark (xi) following Theorem 1.14]{BYY19GeometricStatistics}. For a detailed discussion on the main difficulties in extending the central limit theorem from \cite{BYY19GeometricStatistics} to obtain fine asymptotic results we refer the reader to the end of \cref{sec:Results}. Compared to the deviation results in \cite{ERS15MDPForStabilizingFunctionals}, we are able to provide deviations for more general stationary point processes. We do not cover non-stationary Poisson point processes though. As to stationary Poisson input, we actually recover their result under slightly stronger assumptions on the score function $ξ$. Due to the generality of our results, we omit a precise statement here and refer the reader to \cref{sec:Results}.

Let us briefly outline the structure of this paper. In \cref{sec:Assumptions}, we present the main notions of stabilization and exponentially fast decay of correlations as well as the assumptions needed for our results. It also contains some notations we are using throughout the article. The main findings as well as a sketch of the idea of the proof can be found in \cref{sec:Results}. \Cref{sec:MarkedInput} provides an extension of our results to marked input point processes. In \cref{sec:ExamplesApplications}, we discuss several examples of point processes satisfying exponentially fast decay of correlations. Moreover, we show exemplary how our results can be applied to score functions for different geometric statistics. Concluding, \cref{sec:Proofs} provides the detailed proofs of our theorems presented in \cref{sec:Results}.


\subsection{Main notions and assumptions}\label{sec:Assumptions}

Within this section, we formalize the concepts of stabilization and having exponentially fast decay of correlations. Moreover, we provide the main assumptions necessary for our theorems. Recall that the goal of this article is to investigate the limiting behaviour of the random measure
\begin{equation*}
	μ_n^ξ = \sum_{x∈\cP_n} ξ\bigl(x,\cP_n\bigr) δ_{xn^{-1/d}}
\end{equation*}
for some simple point process $\cP$ and some score function $ξ$. Moreover, recall that $\cP_n = \cP\cap W_n$ with $W_n = \intcc{-\frac{1}{2}n^\frac{1}{d},\frac{1}{2}n^\frac{1}{d}}^d$ being the box of volume $n$. Sometimes, we use $\cP_∞ = \cP$. Denote by $\cN$ the set of locally finite simple point sets in $ℝ^d$. By score function we mean more precisely any function $ξ\colon ℝ^d\times \cN \to ℝ$ satisfying $ξ(x,\cX) = 0$ whenever $x\notin\cX$ which is measurable with respect to the standard $σ$-algebras on the respective spaces. Whenever we evaluate the random measure $μ_n^ξ$ at a test function $f\colon ℝ^d \to ℝ$, the function $f$ will be always measurable and bounded.

Throughout the article we denote indices in $ℕ$ by $i,j,k,p,q,…$ and the corresponding multi-indices by $\mult k = (k_1,…,k_p)∈ℕ^p$, etc. We also use $\abs{\mult k} = \sum_{i=1}^p k_i$, $\mult k! = \prod_{i=1}^p k_i!$ and $\abs{\mult k}! = (\sum_{i=1}^p k_i)!$. Moreover, for any set $I\subseteq \set{1,…,p}$, we denote $\mult k_I = (k_i)_{i∈I}$. Similarly, we denote points in $ℝ^d$ by $x,y,…$ and vectors of such points by $\mult x = (x_1,…,x_p)∈(ℝ^d)^p$, etc. Again, we denote $\mult x_I = (x_i)_{i∈I}$. Let us call $\mult x$ distinct if all its components $x_i$ are distinct elements in $ℝ^d$. For two points $x,y∈ℝ^d$ we denote their Euclidean distance by $\norm{x-y}$. The ball of radius $r∈ℝ_+$ around $x∈ℝ^d$ will be denoted by $B_r(x) = \set{y∈ℝ^d \given \norm{x-y}\le d}$. By $ϑ_d$ we note the volume of the $d$-dimensional unit ball. We further use $\dist(\mult x,\mult y) = \min_{i,j}\norm{x_i-y_j}$ for two collections of points $\mult x∈(ℝ^d)^p$ and $\mult y∈(ℝ^d)^q$ to denote the distance between the two vectors $\mult x$ and $\mult y$. By $r,s,t,…$ we denote real-valued numbers, and $c,C,C_1,…$ will be used for constants in $ℝ_+$, which are usually irrelevant for our results. For any $ℝ$-valued function $f$ we denote by $\norm{f}_∞$ its supremum norm.

Before being able to state all assumptions, we first review some notions from the theory of point processes. For a proper introduction we refer the reader to the text books \cite{DV03IntroductionPointProcessesI,DV08IntroductionPointProcessesII} and \cite{K17RandomMeasures}. As usual, we treat a point process simultaneously as a random measure or as a collection of random points. In particular, for any set $B\subseteq ℝ^d$, we denote by $\cP(B)$ the number of points of $\cP$ in $B$, and for any bounded function $f\colon ℝ^d \to ℝ$ we denote by $\cP(f)$ the integral of $f$ with respect to the random measure $\cP$. The $p$-point correlation function (provided it exists) $ρ^{(p)}\colon (ℝ^d)^p \to ℝ$ of $\cP$ (or sometimes also called joint intensity) is the function satisfying 
\begin{equation*}
	\Ex[\bigg]{\prod_{i=1}^p \cP(B_i)} = \int_{\prod\limits_{\mathclap{1\le i\le p}} B_i} ρ^{(p)}(\mult x) \dif \mult x
\end{equation*}
for any collection of mutually disjoint bounded Borel sets $B_1,…,B_p$ in $ℝ^d$ and vanishing on the diagonals, i.e.\ $ρ^{(p)}(\mult x) = 0$ for $\mult x$ which are not distinct. Roughly speaking, the $p$-point correlation function $ρ^{(p)}(\mult x)$ provides a measure for the probability of finding points in $\cP$ around $x_1,…,x_p$. Provided the $p$-point correlation function exists, one can derive an explicit formula for the $p$-th moment of $\cP(B)$ for some set $B\subseteq ℝ^d$:
\begin{equation*}
	\Ex[\big]{\cP(B)\bigl(\cP(B)-1\bigr)\dotsm \bigl(\cP(B)-p+1\bigr)} = \int_{B^p} ρ^{(p)}(\mult x) \dif \mult x
\end{equation*}
or equivalently
\begin{equation*}
	\Ex[\big]{\cP(B)^p} = \sum_{i=1}^p \stirlingII{p}{i} \int_{B^i} ρ^{(i)}(\mult x) \dif \mult x,
\end{equation*}
where $\stirlingII{p}{i}$ denote the Stirling numbers of second kind. Due to these relations, the correlation functions are sometimes also called factorial moment densities. From measure theoretic induction it follows that similar relations also hold true for bounded, measurable test functions $f\colon (ℝ^d)^p \to ℝ$:
\begin{equation}\label{eq:Campbell}
	\Ex[\bigg]{\smashoperator[r]{\sum_{\substack{\mult x∈\cP^p\\\mult x \text{ distinct}}}} f(\mult x)}
	= \int_{(ℝ^d)^p} f(\mult x) ρ^{(p)}(\mult x) \dif \mult x.
\end{equation}
This formula is known as Campbell-Mecke formula.

To be able to deal with statistics which might depend on the whole point process, we need an extension of the above-mentioned theory. Such an extension is available under the name of Palm theory. One can view \cref{eq:Campbell} in fact as the defining formula for the $p$-point correlation function. Following this approach, for any function $f\colon (ℝ^d)^p \times \cN \to ℝ$, we define the $p$-th Palm measure $\Prr{\mult x}{\cdot}$ for $\mult x∈(ℝ^d)^p$ as the $ρ^{(p)}(\mult x) \dif \mult x$-almost surely unique measure on $\cN$ satisfying the refined Campbell-Mecke formula
\begin{equation}\label{eq:refinedCampbell}
	\Ex[\bigg]{\smashoperator[r]{\sum_{\substack{\mult x∈\cP^p\\\mult x \text{ distinct}}}} f(\mult x;\cP)}
	= \int_{(ℝ^d)^p} \int_{\cN} f(\mult x;μ) \dif \Prr{\mult x}{μ} ρ^{(p)}(\mult x) \dif \mult x.
\end{equation}
For reasons of simplicity, we also define the Palm expectation $\Exx{\mult x}{\cdot}$ as the expectation with respect to the Palm measure. One can intuitively imagine the Palm measure as the distribution of the point process $\cP$ conditioned on having points at $x_1,…,x_p$. When considering the $ξ$-weighted measure $μ^ξ = \sum_{x∈\cP} ξ(x,\cP)δ_{xn^{-1/d}}$ and any test function $f\colon ℝ^d \to ℝ$, the refined Campbell-Mecke formula now extends to
\begin{align*}
	\MoveEqLeft \Ex[\Big]{\bigl(μ^ξ(f)\bigr)^p}
	= \Ex[\bigg]{\Bigl(\sum_{x∈\cP} ξ(x,\cP)f(x)\Bigr)^p}\\
	&= \sum_{(π_1,…,π_k)∈\cQ_p} \int_{ℝ^k} f\bigl(x_1n^{-\frac{1}{d}}\bigr)^{\abs{π_1}} \dotsm f\bigl(x_kn^{-\frac{1}{d}}\bigr)^{\abs{π_k}} \Exx[\bigg]{\mult x}{\prod_{i=1}^k ξ(x_i,\cP)^{\abs{π_i}}}ρ^{(k)}(\mult x) \dif \mult x,
\end{align*}
where $\cQ_p$ denotes the set of all set partitions of $\set{1,…,p}$. This shows that we can interpret
\begin{equation*}
	m_{\mult k}(\mult x) = \Exx[\bigg]{\mult x}{\prod_{i=1}^p ξ(x_i,\cP)^{k_i}} ρ^{(p)}(\mult x)
\end{equation*}
with $\mult k∈ℕ^p$ and $\mult x∈(ℝ^d)^p$ as the correlation function of the $ξ$-weighted measure $μ^ξ$.
For more insight into Palm theory we refer the reader to \cite[Section 13]{DV08IntroductionPointProcessesII} and \cite[Chapter 6]{K17RandomMeasures}. With that, we are ready to state the assumptions needed for our results.

\subsubsection*{Translation invariance}

Throughout the article, we always assume the point process $\cP$ on $ℝ^d$ to be stationary, i.e.\ the translation $\cP+x$ for some $x∈ℝ^d$ has the same distribution as the point process $\cP$ itself. Moreover, we assume the score function to be translation invariant, meaning that for all points $z∈ℝ^d$ and $x∈\cX∈\cN$ it holds $ξ(x+z,\cX+z) = ξ(x,\cX)$. Both properties will always be assumed without further mentioning them explicitly every single time.

\subsubsection*{Exponentially fast decay of correlations}

We say a function $Φ\colon ℝ_+ \to ℝ_+$ is $\hat{a}$-exponentially fast decaying for some parameter $\hat{a}>0$ if
\begin{equation*}
	\limsup_{s→∞} \frac{\log Φ(s)}{s^{\hat{a}}} < 0
\end{equation*}
or, put differently, there exist constants $c,C>0$ such that
\begin{equation*}
	Φ(s) \le Ce^{-cs^{\hat{a}}}.
\end{equation*}
\begin{definition}[EDC: Exponentially fast decay of correlations]\label{def:EDC}
	We say that the point process $\cP$ has exponentially fast decay of correlations with parameters $a∈\intco{0,1}$ and $\hat a>0$, or $\cP$ satisfies $\EDC(a,\hat{a})$ for short, if there exists a constant $C\ge1$ and a continuous, $\hat{a}$-exponentially fast decaying function $Φ\colon ℝ_+ \to ℝ_+$ such that for all $p∈ℕ\setminus\set{1}$, $\mult x∈(ℝ^d)^p$ and all $\emptyset \ne I \subsetneq \set{1,…,p}$ it holds
	\begin{equation}\label{eq:DecayCorrelation}
		\abs[\Big]{ρ^{(p)}(\mult x) - ρ^{(\abs{I})}(\mult x_I)ρ^{(\abs{I^c})}(\mult x_{I^c})}
		\le C^pp!^aΦ\bigl(\dist(\mult x_I,\mult x_{I^c})\bigr).
	\end{equation}
	Here, $\dist(\mult x_I,\mult x_{I^c}) = \min_{i∈I,j∈I^c} \norm{x_i-x_j}$ denotes the distance between the points $(x_i)_{i∈I}$ and $(x_j)_{j∈I^c}$.
\end{definition}
\begin{remark}
	Having exponentially fast decay of correlations is a measure of being close to independence, as a Poisson point process with intensity $κ$ satisfies $ρ^{(p)}(\mult x) = κ^p$ and hence has exponentially fast decaying correlations with $Φ=0$.
\end{remark}

\subsubsection*{Bound on correlation functions}

\begin{definition}[BC: Bound on correlation functions]\label{def:BC}
	We say that the point process $\cP$ satisfies the bound on correlation functions with parameter $α∈\intco{0,1}$, or $\cP$ satisfies $\BC(α)$ for short, if there exists a constant $C\ge1$ such that for all $p∈ℕ$ it holds
	\begin{equation}\label{eq:BoundCorrelation}
		\sup_{\mult x∈(ℝ^d)^p} ρ^{(p)}(\mult x) \le C^p p!^α.
	\end{equation}
\end{definition}
\begin{lemma}
	For the point process $\cP$, $\EDC(a,\hat{a})$ implies $\BC(α)$ with $α\le a$.
\end{lemma}
\begin{proof}
	Since $\cP$ has exponentially fast decaying correlation functions, there exists a continuous, $\hat{a}$-exponentially fast decaying function $Φ$ satisfying \cref{eq:DecayCorrelation}. Clearly, $Φ$ is bounded, by $C_1$ say. Then, by \cref{eq:DecayCorrelation} and stationarity, we obtain
	\begin{equation*}
		ρ^{(p)}(\mult x) \le ρ^{(1)}(0)^p + \sum_{i=1}^p C_1C^i i!^a ρ^{(1)}(0)^{p-i}
		\le C_1 ρ^{(1)}(0)^pC^p p!^ap.\qedhere
	\end{equation*}
\end{proof}

\subsubsection*{Stabilization}

\begin{definition}
	Given a score function $ξ$ and input $x∈\cX∈\cN$, define the radius of stabilization $R^ξ(x,\cX)$ as the smallest radius $r∈ℝ_+$ such that
	\begin{equation*}
		ξ\bigl(x,\cX\cap B_r(x)\bigr) = ξ\bigl(x,\bigl(\cX\cap B_r(x)\bigr)\cup \bigl(\cY\cap B_r(x)^c\bigr)\bigr)
	\end{equation*}
	for all $\cY∈\cN$. If no such $r$ exists, set $R^ξ(x,\cX) = ∞$.
\end{definition}
Notice that $R$ is translation invariant, as $ξ$ is so. The following definition of stabilization requires the radius of stabilization to be small. Therefore, as soon as the points in $\cP_n$ are far away from each other, the summands in \cref{eq:GeomStat} have to be roughly independent.
\begin{definition}[S: Stabilization]\label{def:ST}
	We say that the score function $ξ$ is stabilizing on $\cP$ with parameter $b∈ℝ_+$, or $ξ$ satisfies $\ST(b)$ for short, if $R^ξ$ satisfies the $b$-moment condition
	\begin{equation}\label{eq:Stabilization}
		 \multiadjustlimits{\sup_{1\le n\le ∞}, \sup_{1\le q\le p}, \sup_{\mult x∈W_n^q}} \Exx[\big]{\mult x}{\abs{R(x_1,\cP_n)}^p} \le C^p p!^b
	\end{equation}
	for some constant $C>0$.
\end{definition}
\begin{remark}
	Notice that our definition of stabilization is equivalent to the definition of stabilization in \cite[Definition 1.6]{BYY19GeometricStatistics} with $c=\frac{1}{b}$:
	\begin{equation}\label{eq:Stabilization2}
		\multiadjustlimits{\sup_{1\le n\le ∞}, \sup_{1\le q\le p}, \sup_{\mult x∈W_n^q}} \Prr[\big]{\mult x}{R(x_1,\cP_n)>s} \le C e^{-C's^\frac{1}{b}}
	\end{equation}
	for some constants $C,C'>0$. Indeed, if $R$ satisfies \eqref{eq:Stabilization2}, then
	\begin{equation*}
		\Exx[\big]{\mult x}{\abs{R(x_1,\cP_n)}^p}
		= \int_0^∞ ps^{p-1}\Prr[\big]{\mult x}{R(x_1,\cP_n)>s} \dif s
		\le C(C')^{-bp} Γ(bp+1),
	\end{equation*}
	where $Γ$ denotes the Gamma function and hence $R$ also satisfies \eqref{eq:Stabilization}. On the other hand, if $R$ satisfies \eqref{eq:Stabilization}, then $\Ex{e^{tR^{1/b}}} < ∞$ for $t$ small enough and \eqref{eq:Stabilization2} follows from an exponential Markov inequality.
\end{remark}

\subsubsection*{Moment growth condition}

As we are interested in more precise asymptotic results than a central limit theorem, a bound on the moment growth of the summands is usually necessary.
\begin{definition}[MG: Moment growth]\label{def:MG}
	We say that the score function $ξ$ satisfies the $β$-moment growth condition with $β∈ℝ_+$, or $ξ$ satisfies $\MG(β)$ for short, if there exists a constant $C\ge1$ such that for all $p∈ℕ$
	\begin{equation*}
		\multiadjustlimits{\sup_{1\le n\le ∞}, \sup_{1\le q\le p}, \sup_{\mult x∈W_n^q}} \Exx[\big]{\mult x}{\abs{ξ(x_1,\cP_n)}^p}
		\le C^p p!^β.
	\end{equation*}
\end{definition}
\begin{remark}
	The moment growth condition for $ξ$ is, as already mentioned, not necessary for proving central limit theorems. This is why a similar condition cannot be found in \cite{BYY19GeometricStatistics}. On the other hand, when proving moderate deviations even in the Poisson case, such a condition is usually assumed, compare e.g.\ assumption MGI for $ξ$ in \cite{ERS15MDPForStabilizingFunctionals}.
\end{remark}

\subsubsection*{Power growth condition}

\begin{definition}[PG: Power growth]\label{def:PG}
	We say that the score function $ξ$ satisfies the $(γ_1,γ_2)$-power growth condition with $γ_1,γ_2∈ℝ_+$, or $ξ$ satisfies $\PG(γ_1,γ_2)$ for short, if there exists a constant $C\ge1$ such that for all $\cX∈\cN$, $r>0$, $x∈ℝ^d$ and $k∈ℕ$ it holds
	\begin{equation*}
		\abs[\big]{ξ\bigl(x,\cX\cap B_r(x)\bigr)}\ind_{\cX(B_r(x)) = k}
		\le C\max\set{1,r}^{γ_1}k^{γ_2}.
	\end{equation*}
\end{definition}
\begin{remark}
	The power growth condition in this article is slightly stronger than the one in \cite{BYY19GeometricStatistics}. The results in \cite{BYY19GeometricStatistics} hold under the weaker bound $C^k\max\set{1,r}^k$ instead. Unfortunately, we were not able to obtain the precise bound on cumulants needed under their condition. Nevertheless, almost all typical geometric statistics also satisfy our stronger version.
\end{remark}
Many score functions admit a radius of stabilization which is deterministic. If the investigated point process satisfies exponentially fast decay of correlations, then the power growth condition immediately implies the moment growth condition. A similar argument was also employed in \cite[Section 2.1]{BYY19GeometricStatistics} less explicitly.
\begin{lemma}
	Assume the point process $\cP$ satisfies $\BC(α)$. Moreover, assume that the radius of stabilization of the score function $ξ$ is bounded (i.e.\ $ξ$ satisfies $\ST(0)$). If $ξ$ satisfies $\PG(γ_1,γ_2)$, then $ξ$ also satisfies $\MG(β)$ with $β\le γ_2$.
\end{lemma}
\begin{proof}
	As the radius of stabilization of $ξ$ is bounded, there exists some constant $r>0$ such that $R^ξ \le r$. Let $p,n∈ℕ$ and consider $\mult x∈W_n^q$ for some $q∈\set{1,…,p}$. Then the power growth condition implies the existence of a constant $C_1>0$ (independent of $p$, $q$, $\mult x$ and $n$) such that
	\begin{align*}
		\Exx[\Big]{\mult x}{\abs{ξ(x_1,\cP_n)}^p}
		&= \sum_{k=0}^∞ \Exx[\Big]{\mult x}{\abs[\big]{ξ\bigl(x_1,\cP_n\cap B_r(x_1)\bigr)}^p \ind_{\cP_n(B_r(x_1))=k}}\\
		&\le C_1^p \max\set{1,r}^{pγ_1} \sum_{k=0}^∞ \Exx[\big]{\mult x}{k^{pγ_2}\ind_{\cP_n(B_r(x_1))=k}}\\
		&\le C_1^p \max\set{1,r}^{pγ_1} \Exx[\big]{\mult x}{\cP(B_r(x_1))^{pγ_2}}.
	\end{align*}
	Denote the constant in the bound on correlation functions \eqref{eq:BoundCorrelation} by $C_2$. By using standard calculus for point processes as introduced in the beginning of \cref{sec:Assumptions}, conclude
	\begin{align*}
		\Exx[\big]{\mult x}{\cP(B_r(x_1))^{pγ_2}}
		&\le \Ex[\big]{(\cP(B_r(x_1))+p)^{pγ_2}}\\
		&\le 2^{pγ_2} \Ex[\big]{\cP(B_r(x_1))^{pγ_2}} + (2p)^{pγ_2}\\
		&\le 2^{pγ_2} \sum_{k=0}^{\ceil{pγ_2}} \stirlingII{\ceil{pγ_2}}{k} \int_{B_r(x_1)^k} ρ_α^{(k)}(\mult y) \dif \mult y + (2p)^{pγ_2}\\
		&\le 2^{pγ_2} C_2^{\ceil{pγ_2}} \bigl(\Vol B_r(x_1)\bigr)^{\ceil{pγ_2}} \sum_{k=0}^{\ceil{pγ_2}} \stirlingII{\ceil{pγ_2}}{k} k!^α + (2p)^{pγ_2}.
	\end{align*}
	Apply the standard bound $\stirlingII{n}{k}\le k^{n-k}$ together with $n^n \ge n!$ and \cref{lem:TouchardSum} to obtain
	\begin{equation*}
		\sum_{k=0}^{\ceil{pγ_2}} \stirlingII{\ceil{pγ_2}}{k} k!^α
		\le \sum_{k=0}^{\ceil{pγ_2}} \frac{k!^α}{k!}k^{\ceil{pγ_2}}
		\le C_3^p (pγ_2+2)!
	\end{equation*}
	for some large constant $C_3$ depending on $α$ only. Finally, combine all three bounds and apply $p^p\le 3^p*p!$ to conclude that $ξ$ satisfies the $β$-moment growth condition with some $β\le γ_2$.
\end{proof}


\subsection{Main results}\label{sec:Results}

Recall that we are considering the $ξ$-weighted point measure evaluated at a bounded, measurable test function $f\colon ℝ^d \to ℝ$ given by
\begin{equation*}
	μ_n^ξ(f) = \sum_{x∈\cP_n} ξ\bigl(x,\cP_n\bigr)f\bigl(xn^{-\frac{1}{d}}\bigr).
\end{equation*}
By $\Normaldist_{0,1}$, let us denote a standard Gaussian distributed random variable, i.e.\ $\Pr{\Normaldist_{0,1} \le t} = (2π)^{-1} \int_{-∞}^t \exp\bigl(-\frac{s^2}{2}\bigr) \dif s$, $t∈ℝ$. Further, denote the variance by
\begin{equation*}
	\bigl(σ_n^ξ(f)\bigr)^2 = \Var[\big]{μ_n^ξ(f)}
\end{equation*}
and its limit by
\begin{equation*}
	σ^2(ξ) = \Exx[\big]{0}{ξ^2(0,\cP)} ρ^{(1)}(0) + \int_{ℝ^d} \bigl(m_{1,1}(0,x) - m_1(0)^2\bigr) \dif x
\end{equation*}
with
\begin{align*}
	m_1(x) &= \Exx[\big]{x}{ξ(x,\cP)}ρ^{(1)}(x),\\
	m_{1,1}(x,y) &= \Exx[\big]{(x,y)}{ξ(x,\cP)ξ(y,\cP)}ρ^{(1)}(x)ρ^{(1)}(y).
\end{align*}
Before presenting our main results, let us recall the mean and variance asymptotic from \cite[Theorem 1.12]{BYY19GeometricStatistics}. Even though the original theorem works under slightly weaker assumptions, we omit presenting them here in full generality for the sake of a better presentation.
\begin{theorem}[Mean and variance asymptotic, {\cite{BYY19GeometricStatistics}}]\label{thm:MeanVarianceAsymptotic}
	Let $f\colon ℝ^d\to ℝ$ be a bounded function. Assume that $\cP$ satisfies $\EDC(a,\hat{a})$ and that $ξ$ satisfies $\ST(b)$, $\MG(β)$ and $\PG(γ_1,γ_2)$. Then,
	\begin{equation*}
		\frac{\Ex[\big]{μ_n^ξ(f)}}{n} \xrightarrow[n→∞]{} \Exx{0}{ξ(0,\cP)}ρ^{(1)}(0)\int_{W_1} f(x) \dif x
	\end{equation*}
	and
	\begin{equation*}
		\frac{\Var[\big]{μ_n^ξ(f)}}{n} \xrightarrow[n→∞]{} σ^2(ξ) \int_{W_1} f(x)^2 \dif x.
	\end{equation*}
\end{theorem}
Note that in general $σ^2(ξ) \int_{W_1} f(x)^2 \dif x$ might be vanishing, e.g.\ typically for statistics from random matrix theory. In such cases, the geometric statistic is degenerate and our results do not apply. Proving a lower bound on the growth of the variance is difficult in general and is mostly treated as a separate problem in the literature on geometric statistics. This is why we will always assume the correct volume-order variance growth and not focus on variance lower bounds here and in the examples provided. For further discussion of this issue and some results on how to obtain a variance lower bound for Poisson input, consult \cite[848]{BYY19GeometricStatistics}, \cite[Section 1.4]{ERS15MDPForStabilizingFunctionals} and \cite[Theorem 2.2]{PW08MultivariateNormalApproximation} and the references therein. Techniques to establish volume-order growth of the variance for certain Gibbsian processes can be found in \cite{XY15NormalApproximationGibbsianInput}.

We now present the main results of this article. Let us start with the following Berry-Esseen estimate:
\begin{theorem}[Central limit theorem with Berry-Esseen bound]\label{thm:BerryEsseen}
	Let $f\colon ℝ^d\to ℝ$ be a bounded function. Assume that the point process $\cP$ satisfies $\EDC(a,\hat{a})$ and that the score function $ξ$ satisfies $\ST(b)$, $\MG(β)$ and $\PG(γ_1,γ_2)$. Moreover, assume that $σ^2(ξ)\int_{W_1} f(x)^2 \dif x > 0$. Then, there exists a constant $C>0$ such that
	\begin{equation*}
		\sup_{s∈ℝ} \abs[\Big]{\Pr[\Big]{μ_n^ξ(f) - \Ex[\big]{μ_n^ξ(f)} \le sσ_n^ξ(f)} - \Pr[\Big]{\Normaldist_{0,1}\le s}}
		\le C n^{-\frac{1}{2+4γ}}
	\end{equation*}
	for all $n∈ℕ$ with
	\begin{equation*}
		γ=
		\begin{cases}
			1 + \max\set{γ_2,β} + \frac{d}{(1-a)\hat{a}} + \frac{bd^2}{(1-a)\hat{a}} & \text{if }\frac{(1-a)\hat{a}}{d} \le 1,\\
			1 + \max\set{γ_2,β} + \frac{d}{\hat{a}} + a + bd & \text{if }\frac{(1-a)\hat{a}}{d} \ge 1.
		\end{cases}
	\end{equation*}
\end{theorem}
\begin{remark}
	As already mentioned, we assume in addition that the variance is of volume order. The central limit theorem presented in \cite[Theorem 1.14]{BYY19GeometricStatistics}, in contrast, also applies for variances growing at least like $n^ε$ for some $ε>0$ but does not yield bounds on the speed of convergence.
\end{remark}
We now turn to the following concentration inequality:
\begin{theorem}[Concentration inequality]\label{thm:Concentration}
	Let $f\colon ℝ^d\to ℝ$ be a bounded function. Assume that the point process $\cP$ satisfies $\EDC(a,\hat{a})$ and that the score function $ξ$ satisfies $\ST(b)$, $\MG(β)$ and $\PG(γ_1,γ_2)$. Moreover, assume that $σ^2(ξ)\int_{W_1} f(x)^2 \dif x > 0$. Then, there exists a constant $C>0$ such that
	\begin{equation*}
		\Pr[\Big]{\abs[\Big]{μ_n^ξ(f)-\Ex{μ_n^ξ(f)}} \ge sσ_n^ξ(f)} \le 2 \exp\Bigl(-\frac{1}{4}\min\set[\Big]{\frac{s^2}{2^{1+γ}}, C(ns^2)^\frac{1}{2+4γ}}\Bigr)
	\end{equation*}
	for all $n∈ℕ$ and $s∈ℝ_+$ with
	\begin{equation*}
		γ=
		\begin{cases}
			1 + \max\set{γ_2,β} + \frac{d}{(1-a)\hat{a}} + \frac{bd^2}{(1-a)\hat{a}} & \text{if }\frac{(1-a)\hat{a}}{d} \le 1,\\
			1 + \max\set{γ_2,β} + \frac{d}{\hat{a}} + a + bd & \text{if }\frac{(1-a)\hat{a}}{d} \ge 1.
		\end{cases}
	\end{equation*}
\end{theorem}
Next, let us state our moderate deviation principle. Before doing so, we briefly recall the notion of moderate deviations for convenience (refer to \cite[Section 3.7]{DZ10LargeDeviations} for instance).
\begin{definition}
	A sequence $(X_n)_{n∈ℕ}$ of random variables satisfies a large deviation principle with speed $(a_n)_{n∈ℕ}$ and (good) rate function $I\colon ℝ\to \intcc{0,∞}$ if $I$ is lower semi-continuous and has compact level sets and if for every Borel set $B\subseteq ℝ$ it holds
	\begin{equation*}
		-\inf_{s∈\operatorname{int}(B)} I(s)
		\le \liminf_{n→∞} \frac{1}{a_n} \log \Pr{X_n∈B}
		\le \limsup_{n→∞} \frac{1}{a_n} \log \Pr{X_n∈B}
		\le -\inf_{s∈\cl(B)} I(s).
	\end{equation*}
	We speak of a moderate deviation principle instead of of a large deviation principle if the scaling of the random variables $(X_n)_{n∈ℕ}$ is between the one of an ordinary law of large numbers and the central limit theorem.
\end{definition}
\begin{theorem}[Moderate deviation principle]\label{thm:MDP}
	Let $f\colon ℝ^d\to ℝ$ be a bounded function. Assume that the point process $\cP$ satisfies $\EDC(a,\hat{a})$ and that the score function $ξ$ satisfies $\ST(b)$, $\MG(β)$ and $\PG(γ_1,γ_2)$. Moreover, assume that $σ^2(ξ)\int_{W_1} f(x)^2 \dif x > 0$. Then, for any sequence $(a_n)_{n∈ℕ}$ of real numbers with $\lim_{n→∞} a_n = ∞$ and $\lim_{n→∞} a_nn^{-\frac{1}{2+4γ}} = 0$, the sequence $\Bigl(\frac{μ_n^ξ(f) - \Ex{μ_n^ξ(f)}}{a_nσ_n^ξ(f)}\Bigr)_{n∈ℕ}$ satisfies a moderate deviation principle on $ℝ$ with speed $a_n^2$ and Gaussian rate function $I(s)=\frac{s^2}{2}$, where
	\begin{equation*}
		γ=
		\begin{cases}
			1 + \max\set{γ_2,β} + \frac{d}{(1-a)\hat{a}} + \frac{bd^2}{(1-a)\hat{a}} & \text{if }\frac{(1-a)\hat{a}}{d} \le 1,\\
			1 + \max\set{γ_2,β} + \frac{d}{\hat{a}} + a + bd & \text{if }\frac{(1-a)\hat{a}}{d} \ge 1.
		\end{cases}
	\end{equation*}
\end{theorem}
\begin{remark}
	In contrast to the moderate deviation principle presented in \cite[Theorem 1.4]{ERS15MDPForStabilizingFunctionals}, we are able to treat the more general input class of point processes with exponentially fast decaying correlations. Their result in turn also covers non-stationary Poisson input. For stationary Poisson point processes, we actually recover the whole regime of their moderate deviation principle under the additional assumption of the score functions satisfying the power growth condition.
\end{remark}
\begin{remark}
	In order to improve our results with our method, one would need to improve the bound on cumulants presented in \cref{thm:BoundCumulants} to decrease the value of $γ$. The optimal bound $C^k k!$ and thus $γ=0$ seems (at least for now) out of reach, though, as already in the case of Poisson input in \cite{ERS15MDPForStabilizingFunctionals} such a bound was not achieved.
\end{remark}
\begin{remark}
	It turned out that there is an error in the proof for the bound on cumulants presented in \cite{ERS15MDPForStabilizingFunctionals}. In Lemma 3.4, the quantity $Q(k,κ,ψ)$ should be $2^{k-1} k!^d \int_0^∞ (1+e\norm{κ}_∞ ω t^d)^{k-1} \dif(-ψ)(t)$ instead of $2^{k-1} k!^d \int_0^∞ (1+e\norm{κ}_∞ ω t^d)^{k-1} \dif(-ψ)(t)$ when using the method of proof presented there. This would result in a bound on cumulants $\abs{\innerp{f^{\otimes k}}{c_λ^k}} \le λ C^k \norm{f}_∞^k k!^{1+γ}$ with $γ=d+α+βd$ instead of $γ=1+α+βd$ and hence would yield deviation results only on a smaller scale. This issue can be overcome by bounding the integral more carefully via the coarea formula as presented here in \cref{lem:Coarea,lem:VolBound}.
\end{remark}
Concluding, let us state our Marcinkiewicz-Zygmund-type strong law of large numbers. Notice that for $ε\ge \frac{1}{2}$ the statement is a consequence of the standard strong law of large numbers; the case $ε<\frac{1}{2}$ is not covered by it, though. The case $ε=0$ represents the scaling of the central limit theorem. Hence, our strong law of large numbers bridges the scaling of the usual strong law and the one of the central limit theorem.
\begin{theorem}[Marcinkiewicz-Zygmund-type strong law of large numbers]\label{thm:SLLN}
	Consider a bounded function $f\colon ℝ^d\to ℝ$. Assume that the point process $\cP$ satisfies $\EDC(a,\hat{a})$ and that the score function $ξ$ satisfies $\ST(b)$, $\MG(β)$ and $\PG(γ_1,γ_2)$. Moreover, assume that $σ^2(ξ)\int_{W_1} f(x)^2 \dif x > 0$. Then, for any $ε>0$ it holds that
	\begin{equation*}
		\frac{μ_n^ξ(f) - \Ex[\big]{μ_n^ξ(f)}}{\sqrt{n}^{1+ε}} \xrightarrow[n→∞]{} 0
	\end{equation*}
	almost surely.
\end{theorem}

Let us conclude this section with a short overview of the methods of proof. Our results crucially depend on an explicit bound on cumulants for the test statistic $μ_n^ξ(f)$ presented in \cref{thm:BoundCumulants}. From that, the Berry-Esseen bound (\cref{thm:BerryEsseen}), the concentration inequality (\cref{thm:Concentration}) and the moderate deviation principle (\cref{thm:MDP}) follow by the well-known work of Saulis and Statulevičius (\cite{SS91LimitTheorems}). To achieve the bound on cumulants we basically use a factorial moment expansion developed in \cite{B95FactorialMomentExpansion,BMS97ExpansionForFunctionals} to obtain a Taylor series like expansion for the moments of our statistic in terms of correlation functions. This expansion allows us to prove fast decay of correlations for the correlation functions of the $ξ$-weighted measure $μ_n^ξ$ in \cref{thm:FastDecayXiWeightedMeasure}. Finally, we apply a clustering result for cumulants (\cref{lem:ClusteringCumu}) together with a coarea formula to obtain the bound on cumulants presented.
The idea of the proof is based on \cite{BYY19GeometricStatistics} where the authors showed that the $k$-th cumulant of $μ_n^ξ(f)$ grows linearly in $n$ and concluded a central limit theorem from this bound. To obtain a central limit theorem, no control on the growth in $k$ is needed though and in Remark (xi) following Theorem 1.14 the authors pose the open problem under what conditions a good control in $k$ can be achieved. This question is answered in our article. The main difficulty in doing so is to translate the fast decay of correlations for the $ξ$-weighted measure into a bound on the factorial cumulant measure without loosing track of the growth in $k$. This is mainly achieved by using the more refined clustering lemma together with a coarea formula instead of the elementary approach in \cite{BYY19GeometricStatistics}. Moreover, when proving fast decay of correlations for the $ξ$-weighted measure, the cut-off $t$ in \cref{lem:BoundM} has to be chosen in a more refined way to balance the contribution from the non-Poissonian part and the part with bounded radius of stabilization better.

A similar cumulant-based approach to obtain fine asymptotic results was also employed in \cite{BESY08MDPGeometricProbability,ERS15MDPForStabilizingFunctionals} for stabilizing functionals of Poisson point processes and in \cite{GT18GaussianPolytopes} for the study of Gaussian polytopes, for instance.

\subsection{Extension to marked point processes}\label{sec:MarkedInput}

Our \cref{thm:BerryEsseen,thm:Concentration,thm:MDP,thm:SLLN} extend to input point processes with independent marks as already known from the case of Poisson point processes (refer e.g.\ to \cite{BY05GaussianLimitsGeometricProbability,ERS15MDPForStabilizingFunctionals}). To present the key arguments more clearly, we only discuss the results for marked point processes here and refrain from working with this extended version throughout the whole article. Nevertheless, all the main tools used in the proofs work in exactly the same way, so that the key bound on cumulants in \cref{thm:BoundCumulants} and thus all conclusions extend to this more general framework. Nevertheless, marked point processes are especially useful in the study of geometric statistics, as they allow applications to time-dependent models like the spacial birth-growth model or random packing. For an application to the latter see \cref{sec:Examples}.

Let $(\cM,\cF_\cM,\bP_\cM)$ be a probability space of marks. Given a point process $\cP$, we define the (independently) marked point process as the subset $\breve{\cP} = \set{(x,m) \given x∈\cP,m∈\cM}$ of $ℝ^d\times \cM$ with distribution given by the product law of $\cP$ and $\bP_\cM$. We call $\cP$ the underlying point process of the marked point process $\breve{\cP}$. Alternatively, we might think of $\breve{\cP}$ as the collection of pairs $(X_i,T_i)_{i∈I}$ where $(X_i)_{i∈I}$ denotes the collection of random points of the point process $\cP$ and $(T_i)_{i∈I}$ is a collection of independent $\bP_\cM$-distributed random variables which is also independent from $\cP$. For a more in-depth introduction to marked point processes we refer the reader to \cite[Section 6.4]{DV03IntroductionPointProcessesI} and further to \cite[278--279]{DV08IntroductionPointProcessesII} for marked Palm theory. By convention, we denote marked objects by a breve accent, i.e.\ $\breve{x} = (x,m) ∈ \breve{ℝ}^d = ℝ^d \times \cM$. When using $x$ and $\breve{x}$ in the same context, $x$ should refer to the projection of $\breve{x}$ onto the space coordinate.

Given a marked point process, consider the associated $ξ$-weighted, marked random measure
\begin{equation*}
	μ_n^ξ = \sum_{\breve{x}∈\breve{\cP}_n} ξ\bigl(\breve{x},\breve{\cP}_n\bigr) δ_{xn^{-1/d}}
\end{equation*}
for some score function $ξ\colon \breve{ℝ}^d \times \cN \to ℝ$. Let us now briefly discuss the changes necessary in our assumptions. Translation invariance, the bound on the correlation function (\cref{def:BC}) and exponentially fast decay of correlation functions (\cref{def:EDC}) should hold for the underlying point process $\cP$. The concept of stabilization needs to be extended slightly in the following way: Define the radius of stabilization $R^ξ(\breve{x},\breve{\cP})$ as the smallest radius $r∈ℝ_+$ such that
\begin{equation*}
	ξ\Bigl(\breve{x},\breve{\cX}\cap (B_r(x)\times \cM)\Bigr)
	= ξ\Bigl(\breve{x},\bigl(\breve{\cX}\cap (B_r(x)\times\cM)\bigr)\cup \bigl(\breve{\cY}\cap (B_r(x)^c\times\cM)\bigr)\Bigr)
\end{equation*}
for all marked point processes $\breve{\cY}$. The moment condition in the definition of stabilization (\cref{def:ST}) and in the definition of the moment growth condition (\cref{def:MG}) should now be uniformly over marked points $\mult{\breve{x}}$. Similarly, we assume that the power growth condition (\cref{def:PG}) holds uniform over the markings, i.e.\ there exists a constant $C\ge1$ such that for all marked point sets $\breve{\cX}$, $r>0$, $\breve{x}∈ℝ^d\times\cM$ and $k∈ℕ$ it holds
\begin{equation*}
	\abs[\big]{ξ\bigl(\breve{x},\breve{\cX}\cap (B_r(x)\times\cM)\bigr)}\ind_{\breve{\cX}(B_r(x)\times\cM) = k}
	\le C\max\set{1,r}^{γ_1}k^{γ_2}.
\end{equation*}
Under these extended assumptions all tools used in this article, in particular the Palm approach, the clustering lemma and the factorial moment expansion, work in exactly the same way as presented in \cref{sec:Proofs}. Hence, we regain the same bound on cumulants as in the unmarked case and thus also \cref{thm:BerryEsseen,thm:Concentration,thm:MDP,thm:SLLN}.


\section{Examples and applications}\label{sec:ExamplesApplications}

\subsection{Examples of point processes with exponentially fast decay of correlations}

\subsubsection*{Superposition of independent point processes with exponentially fast decay of correlations}

Given a tuple of independent point processes, their union is again a point process, called the superposition. It turns out that the superposition of point processes with exponentially fast decay of correlations again has exponentially fast decaying correlations. In the case of independent and identically distributed point processes this is stated and proven in \cite[Proposition 2.3]{BYY19GeometricStatistics}. Their proof extends to the case of independent but not necessarily identically distributed point processes.
\begin{proposition}\label{thm:Superposition}
	Let $k∈ℕ$ and $\cP_1,…,\cP_k$ be independent point processes, all satisfying $\EDC(a_i,\hat{a}_i)$ and $\BC(α_i)$, for $i=1,…,k$ respectively. Then, the superposition $\cup_{i=1}^k \cP_i$ satisfies $\EDC(a,\hat{a})$ with parameters $a=\max\set{a_i+\sum_{j\ne i} α_j \given i=1,…,k}$ and $\hat{a}=\min\set{\hat{a}_i \given i=1,…,k}$ as well.
\end{proposition}
\begin{proof}
	The proof works in exactly the same way as the corresponding one in the case of independent and identically distributed processes presented in \cite[Proposition 1.8]{BYY19GeometricStatisticsSupp}.
\end{proof}

\subsubsection*{\texorpdfstring{$α$}{Alpha}-determinantal point processes}

As the leading example of a point process $\cP$ with exponentially fast decay of correlations we consider the class of stationary determinantal point processes. If the kernel $\cK$ of a determinantal point process satisfies $\cK(x,y)\le Φ(\dist(x,y))$ with some continuous, $\hat{a}$-exponentially fast decaying function $Φ$, then $\cP$ satisfies exponentially fast decay of correlations with parameters $a=0$ and $\hat{a}$. In particular, the bound on the correlation function follows by \cref{thm:DPPDecay} with $α=0$ as well.
\begin{example}
	Probably the most classical determinantal point processes is the (infinite, complex) Ginibre point process with kernel $\cK(w,z) = \exp\bigl(\bar{w}z - \frac{\abs{z}^2}{2} - \frac{\abs{w}}{2}\bigr) \le \exp\bigl(-\frac{\abs{z-w}^2}{2}\bigr)$, $w,z∈\bC$ with respect to the complex Lebesgue measure. Hence we can choose $\hat{a}=2$ and $a=0$.
\end{example}

Our results also apply to the more general class of $α$-determinantal point processes. The $α$ here should not be confused with the parameter $α$ in the bound on correlation functions. The processes discussed here turn out to satisfy the $0$-bound on correlation functions. To define $α$-determinantal point processes consider the $α$-determinant $\det_α(A)$ of an $n\times n$-matrix $A=(a_{i,j})_{i,j}$ first introduced by Vere-Jones in \cite{V88GeneralizationPermanentsDeterminants,V97AlphaPermanents} (back then still in a slightly different form) and given by
\begin{equation*}
	\det\nolimits_α(A) = \sum_{τ∈\cS_n} α^{n-ν(τ)} \prod_{i=1}^n a_{i,τ(i)},
\end{equation*}
where $\cS_n$ denotes the symmetric group on $\set{1,…,n}$ and  $ν(τ)$ denotes the number of cycles of a permutation $τ∈\cS_n$. Notice that for $α=-1$ we obtain the standard determinant and for $α=1$ the so-called permanent. Given any Hermitian, positive semi-definite, locally square integrable kernel $\cK\colon Λ\times Λ \to \bC$ on some locally compact Polish space $Λ$, one can define for any $-\frac{1}{α}∈ℕ$ a point process $\cP$ with $p$-point correlation functions $ρ_α^{(p)}$ given by
\begin{equation*}
	ρ_α^{(p)}(x_1,…,x_p) = \det\nolimits_α \bigl(\cK(x_i,x_j)\bigr)_{1\le i,j\le n},\qquad x_1,…,x_p∈Λ.
\end{equation*}
Such a point process is called $α$-determinantal point process. In case $α=-1$, one refers to the corresponding process as a determinantal point process. Further, the case $α=0$ corresponds to the Poisson point process.
\begin{remark}
	One can define a point process in case of $\frac{1}{α}∈ℕ$ as well. These processes are called $α$-permanental point processes or permanental point process in case $α=1$. Unfortunately, we cannot deal with such processes within our framework, as the decay of correlations has parameters of typical order $a=\frac{1}{α}\ge1$, e.g.\ $ρ_1^{(p+1)}(x_1,x_2,…,x_2)-ρ_1^{(1)}(x_1)ρ_1^{(p)}(x_2,…,x_2) = p*p!*\cK(x_1,x_2)^2$, which is growing too fast in $p$. 
\end{remark}
Over the course of the rest of the article, we always focus on the case $-\frac{1}{α}∈ℕ$. It turns out that
\begin{align*}
	ρ_α^{(p)}(x_1,…,x_p) &\le \prod_{i=1}^p ρ_α^{(1)}(x_i)
\end{align*}
for $α<0$, i.e.\ the points of an $α$-determinantal point process repel each other. Actually, one can define $α$-determinantal point processes for more values of $α$. For further information we refer the reader to \cite{HKPV09ZerosOfGAF,M16ExistenceAlphaDeterminantalProcesses,ST03RandomPointFieldsI}. One last result we need for $α$-determinantal point processes is the following decomposition result, which can be found in \cite[Section 4.10]{HKPV09ZerosOfGAF}:
\begin{proposition}\label{thm:AlphaDPP}
	Any $α$-determinantal point process $\cP_α$ for some $-\frac{1}{α}∈ℕ$ with kernel $\cK$ is the superposition of $-\frac{1}{\abs{α}}$ independent and identically distributed copies of a determinantal point process with kernel $\abs{α}\cK$.
\end{proposition}

We now show that $α$-determinantal point processes satisfy exponentially fast decay of correlations.
\begin{proposition}\label{thm:DPPDecay}
	Let $\cP$ be a determinantal point process. Assume that its associated kernel is Hermitian, positive semi-definite, locally square integrable and exponentially fast decaying in the sense that $\cK(x,y) \le Φ(\dist(x,y))$ for all $x,y$ and some $\hat{a}$-exponentially fast decaying function $Φ$. Then, $\cP$ has exponentially fast decay of correlations with parameters $a=0$ and $\hat{a}$.
\end{proposition}
\begin{remark}
	A weaker version of this result has already been proven in \cite{BYY19GeometricStatistics} but with $a=\frac{1}{2}$ in case of determinantal and $a=1$ in case of permanental point processes. Our improvement in the parameter parameter is due to the better factor $n^2$ compared to $n^{1+\frac{n}{2}}$ in \eqref{eq:EDCDPP}. While for obtaining the central limit theorem as in \cite{BYY19GeometricStatistics} it is not necessary to control the value of $a$, our fine asymptotic results rely crutially on the fact that $a<1$. Hence, the worth bound $a=\frac{1}{2}$ for determinantal point processes would not suffice to obtain fine asymptotic results for $α$-determinantal point processes with $α<-1$.
\end{remark}
To prove \cref{thm:DPPDecay} we need the following lemma which can be found in a more advanced setting for infinite operators in \cite{S05TraceIdeals}. For convenience's sake, we present a proof in our finite-dimensional context here.
\begin{lemma}\label{lem:ContinuityDet}
	For any $n\times n$-matrices $A$ and $B$ it holds
	\begin{equation*}
		\abs[\big]{\det(A) - \det(B)} \le \norm[\big]{A-B}_{S_1} e^{\norm{A}_{S_1} + \norm{B}_{S_1}},
	\end{equation*}
	where $\norm{A}_{S_1} = \sqrt{\tr(A^*A)} = \sum_{i=1}^n σ_i(A)$ denotes the Schatten 1-norm and $σ_i(A)$ denotes the $i$-th singular value of $A$.
\end{lemma}
\begin{proof}
	Throughout the proof, denote by $(λ_i)_{i=1,…,n}$ the eigenvalues of $A$. Notice that
	\begin{equation*}
		\det{A} = \prod_{i=1}^n λ_i \le e^{\tr(A)-n} \le e^{\norm{A}_{S_1}-n}.
	\end{equation*}
	Now, define $f(z) = \det\bigl(\frac{1}{2}(A+B)-z(A-B)\bigr)$, $z∈\bC$. Since determinants are continuous, the function $f$ is analytic. Next, apply the mean value theorem together with Cauchy's integral formula to obtain for any $r>0$ that
	\begin{align*}
		\abs[\big]{\det(A)-\det(B)}
		&= \abs[\Big]{f\Bigl(\frac{1}{2}\Bigr) - f\Bigl(-\frac{1}{2}\Bigr)}
		\le \sup_{-\frac{1}{2}\le s\le \frac{1}{2}} \abs[\big]{f'(s)}\\
		&= \sup_{-\frac{1}{2}\le s\le \frac{1}{2}} \abs[\bigg]{\frac{1}{2πi} \oint_{B_r(0)} \frac{f(z+s)}{z^2} \dif z}\\
		&\le \adjustlimits\sup_{-\frac{1}{2}\le s\le \frac{1}{2}} \sup_{z∈B_r(0)} \frac{\abs{f(z+s)}}{r}.
	\end{align*}
	Finally, combine both bounds and evaluate at $r=\norm{A-B}_{S_1}^{-1}$ to conclude
	\begin{equation*}
		\abs[\big]{\det(A)-\det(B)}
		\le r^{-1} \sup_{z∈B_r(0)}e^{\norm{A}_{S_1} + \norm{B}_{S_1} + \abs{z}\norm{A-B}_{S_1} - n}
		\le \norm[\big]{A-B}_{S_1} e^{\norm{A}_{S_1}+\norm{B}_{S_1}}. \qedhere
	\end{equation*}
\end{proof}
\begin{proof}[Proof of {\cref{thm:DPPDecay}}]
	Let $\mult x∈(ℝ^d)^p$ and consider $\emptyset\ne I\subsetneq \set{1,…,p}$. Denote $K = (\cK(x_i,x_j))_{i,j=1,…,p}$, $K_I = (\cK(x_i,x_j))_{i,j∈I}$ and $K_{I^c} = (\cK(x_i,x_j))_{i,j∈I^c}$. Apply \cref{lem:ContinuityDet} to obtain
	\begin{align*}
		\MoveEqLeft \abs[\big]{ρ^{(p)}(\mult x) - ρ^{(\abs{I})}(\mult x_I)ρ^{(\abs{I^c})}(\mult x_{I^c})}\\
		&= \abs[\big]{\det(K) - \det(K_I)\det(K_{I^c})}\\
		&\le \norm[\Big]{K-\begin{psmallmatrix}K_I&0\\0&K_{I^c}\end{psmallmatrix}}_{S_1} \exp\Bigl(\norm{K}_{S_1}+\norm[\Big]{\begin{psmallmatrix}K_I&0\\0&K_{I^c}\end{psmallmatrix}}_{S_1}\Bigr).
	\end{align*}
	Finally, recall that for a Hermitian, positive semi-definite matrix $A$ the Schatten 1-norm is given by $\norm{A}_{S_1} = \tr(A)$. Hence, $\norm{K}_{S_1} = \tr(K) \le n\norm{\cK}_∞$ and similarly $\norm[\big]{\begin{psmallmatrix}K_I&0\\0&K_{I^c}\end{psmallmatrix}}_{S_1} = \norm{K_I}_{S_1} + \norm{K_{I^c}}_{S_1} = \tr(K) \le n\norm{\cK}_∞$. Concerning the remaining term, notice that for any $n\times n$-matrix $A$ with all entries bounded by $a$ it holds that $\norm{A}_{S_1} \le n^2 a$, which can be established by applying the Gershgorin circle theorem, for instance. Hence, conclude
	\begin{equation}\label{eq:EDCDPP}
		\abs[\big]{ρ^{(p)}(\mult x) - ρ^{(\abs{I})}(\mult x_I)ρ^{(\abs{I^c})}(\mult x_{I^c})}
		\le n^2*Φ\bigl(\dist(\mult x_I,\mult x_{I^c})\bigr) e^{n\norm{\cK}_∞}.
	\end{equation}
	This shows that the correlation functions decay exponentially fast with parameters $a=0$ and $\hat{a}$ given by $Φ$.
\end{proof}
We now combine \cref{thm:AlphaDPP,thm:DPPDecay} with \cref{thm:Superposition} to conclude the exponentially fast decay of correlations for $α$-determinantal point processes.
\begin{corollary}
	Let $\cP$ be an $α$-determinantal point process with $-\frac{1}{α}∈ℕ$ and exponentially fast decaying kernel $\cK$ with parameter $\hat{a}$. Then $\cP$ has exponentially fast decay of correlations with parameters $a=0$ and $\hat{a}$. 
\end{corollary}

\subsubsection*{Rarified Gibbsian input}

In \cite{SY13GeometricFunctionalsGibbs}, the authors show that some classes of Gibbsian point processes admit exponentially fast decay of correlations. Consider a homogeneous Poisson point process $\cP_λ$ with intensity $λ>0$, an inverse temperature $β>0$ and a Hamiltonian $H$ from one of the following classes:
\begin{enumerate}
	\item Pair potential functions: $H(\cX) = \sum_{x,y∈\cX,x\ne y} φ(\norm{x-y})$ with $φ\colon ℝ_+ \to ℝ_+$ having compact support or satisfying the superstability condition $φ(s) \le C_1 \exp(-C_2 s)$ for some constants $C_1,C_2>0$ on $s∈\intco{s_0,∞}$, $s_0∈ℝ_+$, together with the hard-core exclusion $φ(s)=∞$ on $s∈\intoo{0,s_0}$.
	\item Area interactions: Given some compact, convex set $K$, consider the Hamiltonian given by $H(\cX) = \Vol(\cup_{x∈\cX} (x+K)) + C_1 \abs{\cX} + C_2$ for some constants $C_1,C_2∈ℝ_+$.
	\item Hard-core potential: Given some parameter $s_0>0$, consider $H(\cX) = ∞$ if there are two points within $\cX$ of distance smaller than $2s_0$ and $H(\cX) = C_1\abs{\cX} + C_2$ for some constants $C_1,C_2∈ℝ_+$ otherwise.
	\item Truncated Poisson: Given a constraint event $E$, consider $H(\cX) = 0$ if $\cX$ satisfies $E$ and $H(\cX) = ∞$ if $\cX$ does not satisfy $E$. You might think of $E$ as having no two points at a distance smaller than a fixed constant.
\end{enumerate}
One can define the Gibbs point process $\cP_λ^{βH}$ then as the process with Radon-Nikodym derivative given by
\begin{equation*}
	\diff{\bigl(\cP_λ^{βH}\cap D\bigr)}{\bigl(\cP_λ\cap D\bigr)}(\cX) = \frac{\exp\bigl(-βH(\cX\cap D)\bigr)}{\Ex[\big]{\exp\bigl(-βH(\cX\cap D)\bigr)}}
\end{equation*}
for all open, bounded subsets $D\subseteq ℝ^d$ and $\cX$ finite.
\begin{proposition}[{\cite[Lemma 3.4]{SY13GeometricFunctionalsGibbs}}]
	The Gibbs point process $\cP_λ^{βH}$ admits exponentially fast decay of correlations with parameters $a=0$ and $\hat{a} = 1$ whenever the intensity $λ$ is chosen sufficiently small (depending explicitly on $β$ and $H$).
\end{proposition}


\subsection{Examples of admissible statistics}\label{sec:Examples}

Our results apply to a wide variety of geometric statistics, including various statistics of $k$-nearest neighbour graphs, random geometric graphs, sphere of influence graphs, germ-grain models, random sequential absorption, spacial birth-growth processes and simplicial complexes. All these examples have been extensively discussed in the literature already. Nevertheless, our asymptotic results seem to be new for all of them under determinantal or Gibbsian input, for instance. For a detailed description of the different models, we refer the reader to the arXiv version of \cite[Section 2.3]{BYY19GeometricStatistics} and to \cite[Section 2]{ERS15MDPForStabilizingFunctionals} and the references therein. Here, we merely provide a quick introduction and application of our results to germ-grain models and random sequential absorption to illustrate how geometric functionals may look like.

\subsubsection*{\texorpdfstring{$k$}{k}-covered region in the germ-grain model}

Consider a locally finite collection $\cX\subseteq ℝ^d$ of so-called germs and a collection of compact sets $S_x\subseteq ℝ^d$, $x∈\cX$, of so-called grains. For us, the grains will all be given by $B_r(x)$ for some fixed $r∈ℝ_+$, but different choices are applicable as well. The goal is to investigate the volume of the $k$-covered region
\begin{equation*}
	\Vol\bigl(\set[\big]{y∈W_n \given \cX(B_r(y)) \ge k}\bigr).
\end{equation*}
This statistic is quite classical in the literature on geometric statistics. A further introduction and applications to real-world phenomena can be found in \cite{H88CoverageProcesses}. Moreover, different statistics such as intrinsic volumes of the germ-grain model are also discussed there. It turns out that the volume of the $k$-covered region can be decomposed into a sum of score functions
\begin{equation*}
	ξ^{(k)}(x,\cX) = \int_{B_r(x)} \frac{\ind_{\cX(B_r(y))\ge k}}{\cX(B_r(y))} \dif y.
\end{equation*}
Clearly, the radius of stabilization of the score function $ξ^{(k)}$ is bounded by $2r$. Also $ξ^{(k)}$ is bounded itself by $r^dϑ_d$. Hence, $ξ^{(k)}$ is stabilizing with $b=0$,  satisfies the moment growth condition $\MG(0)$ and the power growth condition $\PG(0,0)$. Thus, \cref{thm:BerryEsseen,thm:Concentration,thm:MDP,thm:SLLN} apply for example for the above-mentioned determinantal and Gibbsian point processes as soon as the limiting variance is non-zero. Hence, we add new asymptotic results to the known central limit theorem \cite[Theorem 2.4]{BYY19GeometricStatistics} under input with fast decaying correlation functions.

\subsubsection*{Random sequential absorption}

The basic random sequential absorption model goes as follows: Consider a point process $\cP$ on $ℝ^d$. We think again of $\cP$ as a collection of random points $(X_i)_{i∈I}$. To each point $X_i$ associate a random mark (or time stamp) $T_i$ uniformly distributed on $\intcc{0,1}$ and independent from all other randomness. These marks are used to establish a chronological order among the points. For reasons of simplicity, we denote the marked point process again by $\breve{\cP}$. To each point we further associate a ball of fixed radius $r∈ℝ_+$. Call the first point, i.e.\ the one with the smallest mark, accepted. Recursively, call any next point accepted if its associated ball does not overlap with any ball from a previously accepted point. Otherwise, we call the point rejected. The goal is to study the asymptotic behaviour of the number of accepted points. This basic random sequential absorption model can also be generalized in several ways. For example one might consider random and time-dependent radii. For a more in-depth discussion of the model and its applications to physics, chemistry and biology we refer to \cite{BY03GaussianFields,PY02LimitTheoryRandomSequentialPacking}. Clearly, the statistic of interest can be decomposed into a sum of score functions $ξ$ with $ξ(\breve{x},\breve{\cP})$ being $1$ if the point $\breve{x}$ is accepted and $0$ if it is rejected. Since $ξ$ is bounded, it immediately satisfies the moment growth condition $\MG(0)$ and the power growth condition $\PG(0,0)$. We refer the reader to \cite[Lemma 4.2]{PY02LimitTheoryRandomSequentialPacking} for the general construction of the radius of stabilization. For determinantal point processes as mentioned in \cref{sec:Examples}, this lemma can be easily extended by using the appropriate void probabilities established e.g.\ in \cite{BYY19GeometricStatisticsSupp}. Hence, as soon as the limiting variance is non-zero, all our asymptotic results apply and extend the known results for random sequential absorption from Poisson input to our more general one. For a general point process with exponentially fast decaying correlation functions, the question of stability has to be checked case by case.


\section{Proofs} \label{sec:Proofs}

The proof is divided into two parts. In the first one, we show that the correlation functions of the $ξ$-weighted measure $μ_n^ξ$ defined by
\begin{equation}\label{eq:CorrelationsXiWeighted}
	m_{\mult k}(\mult x;n) = \Exx[\bigg]{\mult x}{\prod_{i=1}^p ξ(x_i,\cP_n)^{k_i}}ρ^{(p)}(\mult x)
\end{equation}
for $n∈ℕ$, $\mult k∈ℕ^p$ and $\mult x∈(ℝ^d)^p$ decay fast, i.e.\ for any $\emptyset\ne I\subsetneq \set{1,…,p}$ we prove an explicit bound for the term $\abs{m_{\mult k}(\mult x;n) - m_{\mult k_I}(\mult x_I;n)m_{\mult k_{I^c}}(\mult x_{I^c};n)}$ which decays in $\dist(\mult x_I,\mult x_{I^c})$ (see \cref{thm:FastDecayXiWeightedMeasure}). In the second part, we use this estimate to deduce the explicit bound on cumulants (\cref{thm:BoundCumulants}) and to conclude our main results.

\subsection{Proof of fast decay of correlations for the \texorpdfstring{$ξ$}{xi}-weighted measure}\label{sec:ProofsDecayCorrelations}

To show that the correlation functions of the $ξ$-weighted measure $μ_n^ξ$ defined in \cref{eq:CorrelationsXiWeighted} decay fast, we compare them to a truncated version with bounded radius of stabilization. This allows us to apply a factorial moment expansion argument. Fix some value $t∈ℝ_+$ (to be determined later) and introduce the following truncated version of $ξ$ and $m_{\mult k}$:
\begin{align*}
	\tilde{ξ}(x,\cP_n)
	&= ξ\bigl(x_i,\cP_n\cap B_{R(x_i,\cP_n)}(x)\bigr) \ind_{\set{R(x_i,\cP_n)\le t}},\\
	\tilde{m}_{\mult k}(\mult x;n)
	&= \Exx[\bigg]{\mult x}{\prod_{i=1}^p \tilde{ξ}(x_i,\cP_n)}ρ^{(p)}(\mult x).
\end{align*}
\begin{lemma}\label{lem:BoundM}
	Assume that the point process $\cP$ satisfies $\BC(α)$ and that the score function $ξ$ satisfies $\ST(b)$ and $\MG(β)$. Then, there exists a constant $C\ge1$ such that for all $p∈ℕ$, $\mult k∈ℕ^p$, distinct $\mult x∈(ℝ^d)^p$, $n∈ℕ$, $\emptyset \ne I\subsetneq \set{1,…,p}$, $t∈ℝ_+$ and $N∈ℕ$ it holds
	\begin{align*}
		\MoveEqLeft \abs[\big]{m_{\mult k}(\mult x;n) - m_{\mult k_I}(\mult x_I;n) m_{\mult k_{I^c}}(\mult x_{I_c};n)}\\
		&\le C^{\abs{\mult k} + N} p!^α \abs{\mult k}!^β \frac{N!^b}{t^N}
		+ \abs[\big]{\tilde{m}_{\mult k}(\mult x;n) - \tilde{m}_{\mult k_I}(\mult x_I;n) \tilde{m}_{\mult k_{I^c}}(\mult x_{I^c};n)}.
	\end{align*}
\end{lemma}
\begin{proof}
	Hölder's inequality together with the $β$-moment growth condition for $ξ$ and the $α$-bound on the correlation functions of $\cP$ imply the existence of two constants $C_1$ and $C_2$ such that
		\begin{align*}
			\abs[\big]{m_{\mult k}(\mult x;n)}
			&= \abs[\bigg]{\Exx[\bigg]{\mult x}{\prod_{i=1}^p ξ^{k_i}(x_i,\cP_n)} ρ^{(p)}(\mult x)}\\
			&\le \prod_{i=1}^p \Exx[\Big]{\mult x}{\abs[\big]{ξ(x_i,\cP_n)}^{\abs{\mult k}}}^\frac{k_i}{\abs{\mult k}} \abs[\big]{ρ^{(p)}(\mult x)}\\
			&\le (C_1C_2)^{\abs{\mult k}} p!^α \abs{\mult k}!^β.
		\end{align*}
	We hence proved the existence of a constant $C_3\ge1$ (independent of $p$, $\mult k$, $\mult x$ and $n$) such that
	\begin{equation}\label{eq:BddXiWeightedMeasure}
		\abs[\big]{m_{\mult k}(\mult x;n)}
		\le C_3^{\abs{\mult k}} p!^α \abs{\mult k}!^β.
	\end{equation}
	This estimate together with another application of Hölder's inequality yields that
	\begin{align*}
		\MoveEqLeft \abs[\big]{m_{\mult k}(\mult x;n) - \tilde{m}_{\mult k}(\mult x;n)}\\
		&= \abs[\bigg]{\Exx[\bigg]{\mult x}{\prod_{i=1}^p ξ^{k_i}(x_i,\cP_n) \ind_{\set{∃j\colon R^ξ(x_j,\cP_n)>t}}} ρ^{(p)}(\mult x)}\\
		&\le \Exx[\bigg]{\mult x}{\prod_{i=1}^p \abs[\big]{ξ(x_i,\cP_n)}^{2k_i}}^\frac{1}{2} \Prr[\big]{\mult x}{∃j\colon R^ξ(x_j,\cP_n)>t}^\frac{1}{2} \abs[\big]{ρ^{(p)}(\mult x)}\\
		&\le (2C_3)^{\abs{\mult k}} p!^α \abs{\mult k}!^β \biggl(\sum_{i=1}^p \Prr[\big]{\mult x}{R^ξ(x_i,\cP_n)>t} \biggr)^\frac{1}{2} .
	\end{align*}
	Next, by the $b$-moment condition for $R$, there exists a constant $C_4$ such that further
	\begin{align*}
		\abs[\big]{m_{\mult k}(\mult x;n) - \tilde{m}_{\mult k}(\mult x;n)}
		&\le (2C_3)^{\abs{\mult k}} p!^α \abs{\mult k}!^β \Biggl( \sum_{i=1}^p \frac{\Exx[\big]{\mult x}{\abs{R^ξ(x_i,\cP_n)}^{2N}}}{t^{2N}} \Biggr)^\frac{1}{2}\\
		&\le (2C_3)^{\abs{\mult k}} p!^α \abs{\mult k}!^β p^\frac{1}{2} C_4^N2^{bN} \frac{N!^b}{t^N}
	\end{align*}
	for any $N∈ℕ$. Hence, we proved that there exists a constant $C\ge1$ (independent of $p$, $\mult k$, $\mult x$, $n$, $N$ and $t$) such that
	\begin{equation*}
		\abs[\big]{m_{\mult k}(\mult x;n) - \tilde{m}_{\mult k}(\mult x;n)}
		\le C^{\abs{\mult k}+N} p!^α \abs{\mult k}!^β \frac{N!^b}{t^N}.
	\end{equation*}
	Combine all derived inequalities and use $\abs{AB-\tilde{A}\tilde{B}} \le \abs{A}\abs{B-\tilde{B}} + \abs{B}\abs{A-\tilde{A}}$ for $\abs{\tilde{B}}\le \abs{B}$ to finally obtain
	\begin{align*}
		\MoveEqLeft \abs[\big]{m_{\mult k}(\mult x;n) - m_{\mult k_I}(\mult x_I;n) m_{\mult k_{I^c}}(\mult x_{I^c};n)}\\
		&\le \abs[\big]{m_{\mult k}(\mult x;n) - \tilde{m}_{\mult k}(\mult x;n)} + \abs[\big]{\tilde{m}_{\mult k}(\mult x;n) - \tilde{m}_{\mult k_I}(\mult x_I;n) \tilde{m}_{\mult k_{I^c}}(\mult x_{I^c};n)}\\
		&\qquad + \abs[\big]{m_{\mult k_I}(\mult x_I;n)} \abs[\big]{m_{\mult k_{I^c}}(\mult x_{I^c};n) - \tilde{m}_{\mult k_{I^c}}(\mult x_{I^c};n)}\\
		&\qquad + \abs[\big]{m_{\mult k_{I^c}}(\mult x_{I^c};n)} \abs[\big]{m_{\mult k_I}(\mult x_I;n) - \tilde{m}_{\mult k_I}(\mult x_I;n)}\\
		&\le \biggl( C^{\abs{\mult k}+N} p!^α \abs{\mult k}!^β \frac{N!^b}{t^N} \biggr) + \abs[\big]{\tilde{m}_{\mult k}(\mult x;n) - \tilde{m}_{\mult k_I}(\mult x_I;n) \tilde{m}_{\mult k_{I^c}}(\mult x_{I^c};n)}\\
		&\qquad + \Bigl( C^{\abs{\mult k_I}}\abs{I}!^α \abs{\mult k_I}!^β \Bigr) \biggl( C^{\abs{\mult k_{I^c}}+N} \abs{I^c}!^α \abs{\mult k_{I^c}}!^β \frac{N!^b}{t^N} \biggr)\\
		&\qquad + \Bigl( C^{\abs{\mult k_{I^c}}} \abs{I^c}!^α \abs{\mult k_{I^c}}!^β \Bigr) \biggl( C^{\abs{\mult k_I}+N} \abs{I}!^α \abs{\mult k_I}!^β \frac{N!^b}{t^N} \biggr)\\
		&\le 3 C^{\abs{\mult k} + N} p!^α \abs{\mult k}!^β \frac{N!^b}{t^N} + \abs[\big]{\tilde{m}_{\mult k}(\mult x;n) - \tilde{m}_{\mult k_I}(\mult x_I;n) \tilde{m}_{\mult k_{I^c}}(\mult x_{I^c};n)}.
	\end{align*}
	This concludes the proof.
\end{proof}
\begin{lemma}\label{lem:TouchardSum}
	For any constants $a∈\intco{0,1}$, $ν∈ℕ_0$ and $s∈ℝ_+$ it holds
	\begin{equation*}
		\sum_{k=0}^∞ \frac{k!^a}{k!}k^ν s^k
		\le \frac{2\max\set[\big]{1,\frac{1}{s}}}{(1-a)^{ν+1}} e^{s^\frac{1}{1-a}} T_{ν+1}\Bigl(s^\frac{1}{1-a}\Bigr)
		\le \frac{2e^{ν+1}(ν+1)!}{(1-a)^{ν+1}}e^{2s^\frac{1}{1-a}},
	\end{equation*}
	with $T_ν$ denoting the $ν$-th Touchard polynomial.
\end{lemma}
\begin{proof}
	For any $k∈ℕ$, Stirling's formula yields
	\begin{equation*}
		\frac{k!^a}{k!}
		\le \biggl(\sqrt{2πk}\Bigl(\frac{k}{e}\Bigr)^k\biggr)^{a-1}
		\le \biggl( \sqrt{2π\floor{(1-a)k}} \Bigl(\frac{\floor{(1-a)k}}{e}\Bigr)^{\floor{(1-a)k}} \biggr)^{-1}
		\le \frac{2}{\floor{(1-a)k}!}.
	\end{equation*}
	Thus,
	\begin{equation*}
		\sum_{k=0}^∞ \frac{k!^a}{k!}k^ν s^k
		\le 2\sum_{k=0}^∞ \frac{k^νs^k}{\floor{(1-a)k}!}
		\le 2\adjustlimits \sum_{l∈ℕ_0} \sum_{\floor{(1-a)k}=l} \frac{k^νs^k}{\floor{(1-a)k}!}.
	\end{equation*}
	Now, $\floor{(1-a)k}=l$ implies $\frac{l}{1-a}\le k\le \frac{l+1}{1-a}$, and thus there are at most $\frac{1}{1-a}$ integers $k$ satisfying $\floor{(1-a)k}=l$. Hence, $s^k \le \max\set[\big]{1,\frac{1}{s}} s^\frac{l+1}{1-a}$. We therefore conclude
	\begin{align*}
		\sum_{k=0}^∞ \frac{k!^a}{n!}k^ν s^k
		&\le \frac{2\max\set[\big]{1,\frac{1}{s}}}{1-a}\sum_{l=0}^∞ \frac{\bigl(\frac{l+1}{1-a}\bigr)^ν s^\frac{l+1}{1-a}}{l!}
		= \frac{2\max\set[\big]{1,\frac{1}{s}}}{(1-a)^{ν+1}} \sum_{l=0}^∞ \frac{(l+1)^{ν+1} s^\frac{l+1}{1-a}}{(l+1)!}\\
		&= \frac{2\max\set[\big]{1,\frac{1}{s}}}{(1-a)^{ν+1}} \sum_{l=0}^∞ \frac{l^{ν+1} s^\frac{l}{1-a}}{l!}
		= \frac{2\max\set[\big]{1,\frac{1}{s}}}{(1-a)^{ν+1}} e^{s^\frac{1}{1-a}} T_{ν+1}\Bigl(s^\frac{1}{1-a}\Bigr),
	\end{align*}
	where $T_ν$ denotes the $ν$-th Touchard polynomial.
	Finally, observe that the bound $\stirlingII{ν}{k} \le \binom{ν}{k}k^{ν-k}$ for $ν\ge 2$ implies
	\begin{align*}
		T_{ν+1}(s)
		&= \sum_{k=1}^{ν+1} \stirlingII{ν+1}{k} s^k
		= s \sum_{k=0}^ν \stirlingII{ν+1}{k+1}s^k
		\le s \sum_{k=0}^ν \binom{ν+1}{k+1} (k+1)^{ν-k} s^k\\
		&\le s \sum_{k=0}^ν \binom{n}{k} (ν+1)^{ν-k} s^k
		= s(ν+s+1)^ν.
	\end{align*}
	For $ν=0$ the bound follows trivially as $T_1(s) = s$. Thus,
	\begin{equation*}
		\max\set[\Big]{1,\frac{1}{s}}T_{ν+1}\Bigl(s^\frac{1}{1-a}\Bigr)
		\le \Bigl(ν+s^\frac{1}{1-a}+1\Bigr)^{ν+1}
		\le (ν+1)!e^{ν+s^\frac{1}{1-a}+1}.\qedhere
	\end{equation*}
\end{proof}
Let us now state the factorial moment expansion from \cite[Lemma 3.2]{BYY19GeometricStatistics} which is basically an adaptation of the factorial moment expansion provided in \cite{B95FactorialMomentExpansion,BMS97ExpansionForFunctionals}. We state it here without proof.
\begin{lemma}[{\cite[Lemma 3.2]{BYY19GeometricStatistics}}]\label{lem:FMEexpansion}
	Assume that the score function $ξ$ satisfies the condition $\PG(γ_1,γ_2)$. Then, for any $p∈ℕ$, $\mult k∈ℕ^p$, distinct $\mult x∈(ℝ^d)^p$, $n∈ℕ$, $\emptyset\ne I\subsetneq \set{1,…,p}$ and $t\le \frac{\dist(\mult x_I,\mult x_{I^c})}{2}$ it holds
	\begin{align*}
		\MoveEqLeft \tilde{m}_{\mult k}(\mult x;n) - \tilde{m}_{\mult k_I}(\mult x_I;n) \tilde{m}_{\mult k_{I^C}}(\mult x_{I^c};n)\\
		&= \sum_{l=0}^∞\sum_{j=0}^l \frac{1}{j!(l-j)!} \int_{(\cup_{i∈I} B_{t,n}(x_i))^j} \dif \mult y \int_{(\cup_{i∈I^c} B_{t,n}(x_{i}))^{l-j}} \dif \mult z\\
		&\qquad \sum_{\substack{J_1\subseteq \set{1,…,j}\\J_2\subseteq \set{1,…,l-j}}} (-1)^{l-\abs{J_1}-\abs{J_2}} \tilde{ψ}_{\mult k}^! \biggl(\mult x_I;\sum_{i∈J_1}δ_{y_i}\biggr) \tilde{ψ}_{\mult k}^! \biggl(\mult x_{I^c};\sum_{i∈J_2}δ_{z_i}\biggr)\\
		&\qquad \biggl( ρ^{(p+l)}(\mult x,\mult y,\mult z) - ρ^{(\abs{I}+j)}(\mult x_I,\mult y)ρ^{(\abs{I^c}+l-j)}(\mult x_{I^c},\mult z) \biggr),
	\end{align*}
	with
	\begin{equation*}
		\tilde{ψ}_{\mult k}^!(\mult x;μ) = \prod_{i=1}^p ξ\Bigl(x_i,μ+\sum_{i=1}^p δ_{x_i}\Bigr)^{k_i}
	\end{equation*}
	for any measure $μ$, any $\mult x∈(ℝ^d)^p$ and $\mult k∈ℕ^p$.
\end{lemma}
\begin{lemma}\label{lem:BoundMTilde}
	Assume that the point process $\cP$ satisfies $\EDC(a,\hat{a})$ and that the score function $ξ$ satisfies $\PG(γ_1,γ_2)$. Then, there exists a constant $C\ge 1$ such that for all $p∈ℕ$, $\mult k∈ℕ^p$, distinct $\mult x∈(ℝ^d)^p$, $n∈ℕ$, $\emptyset \ne I \subsetneq \set{1,…,p}$ and $t \le \frac{\dist(\mult x_I,\mult x_{I^c})}{2}$ it holds
	\begin{align*}
		\MoveEqLeft \abs[\big]{\tilde{m}_{\mult k}(\mult x;n) - \tilde{m}_{\mult k_I}(\mult x_I;n) \tilde{m}_{\mult k_{I^c}}(\mult x_{I^c};n)}\\
		&\le C^{\abs{\mult k}} p!^a \abs{\mult k}!^{γ_2} Φ\bigl(\dist(\mult x_I,\mult x_{I^c})-2t\bigr) e^{C\abs{\mult k}^\frac{1}{1-a}t^\frac{d}{1-a}}.\qedhere
	\end{align*}
\end{lemma}
\begin{proof}
	Firstly, notice that the $(γ_1,γ_2)$-power growth condition yields
	\begin{equation*}
		\abs[\bigg]{\tilde{ψ}_{\mult k}^! \biggl(\mult x_I,\sum_{i∈J_1} δ_{y_i}\biggr)}
		= \prod_{i∈I}\abs[\bigg]{ξ\biggl(x_i;\sum_{j∈J_1}δ_{y_j} + \sum_{j∈I}δ_{x_j}\biggr)}^{k_i}
		\le C_1^{\abs{\mult k_I}} \max\set[\big]{1,t}^{γ_1\abs{\mult k_I}} \bigl(j+\abs{I}\bigr)^{γ_2\abs{\mult k_I}}
	\end{equation*}
	and similarly
	\begin{equation*}
		\abs[\bigg]{\tilde{ψ}_{\mult k}^! \biggl(\mult x_{I^c};\sum_{i∈J_2}δ_{z_i}\biggr)}
		\le C_1^{\abs{\mult k_{I^c}}} \max\set[\big]{1,t}^{γ_1\abs{\mult k_{I^c}}} \bigl(l-j+\abs{I^c}\bigr)^{γ_2\abs{\mult k_{I^c}}}
	\end{equation*}
	for some constant $C_1>1$.
	Summing over all subsets $J_1$ and $J_2$ thus implies
	\begin{align*}
		\MoveEqLeft \abs[\Bigg]{\sum_{\substack{J_1\subseteq \set{1,…,j}\\J_2\subseteq \set{1,…,l-j}}} (-1)^{l-\abs{J_1}-\abs{J_2}} \tilde{ψ}_{\mult k}^! \biggl(\mult x_I;\sum_{i∈J_1}δ_{y_i}\biggr) \tilde{ψ}_{\mult k}^! \biggl(\mult x_{I^c};\sum_{i∈J_2}δ_{z_i}\biggr)}\\
		&\le 2^l C_1^{\abs{\mult k}} \max\set[\big]{1,t}^{γ_1\abs{\mult k}} \bigl(j+\abs{I}\bigr)^{γ_2\abs{\mult k_I}} \bigl(l-j+\abs{I^c}\bigr)^{γ_2\abs{\mult k_{I^c}}}.
	\end{align*}
	From the choice of $\mult y$ and $\mult z$, the exponentially fast decay of correlations of $\cP$ implies
	\begin{align*}
		\MoveEqLeft \abs[\big]{ρ^{(p+l)}(\mult x,\mult y,\mult z) - ρ^{(\abs{I}+j)}(\mult x_I,\mult y) ρ^{(\abs{I^c}+l-j)}(\mult x_{I^c},\mult z)}\\
		&\le C_2^{p+l} (p+l)!^a Φ\bigl(\dist((\mult x_I,\mult y),(\mult x_{I^c},\mult z))\bigr)\\
		&\le C_2^{p+l} (p+l)!^a Φ\bigl(\dist(\mult x_I,\mult x_{I^c})-2t\bigr)
	\end{align*}
	for some constant $C_2>1$. Recall that $ϑ_d$ denotes the volume of the $d$-dimensional unit ball. Combining the above-mentioned estimates with \cref{lem:FMEexpansion} allows us to conclude
	\begin{align*}
		\MoveEqLeft \abs[\big]{\tilde{m}_{\mult k}(\mult x;n) - \tilde{m}_{\mult k_I}(\mult x_I;n) \tilde{m}_{\mult k_{I^c}}(\mult x_{I^c};n)}\\
		&= \Biggl| \sum_{l=0}^∞\sum_{j=0}^l \frac{1}{j!(l-j)!} \int_{(\cup_{i∈I} B_{t,n}(x_i))^j} \dif \mult y \int_{(\cup_{i∈I^c} B_{t,n}(x_{i}))^{l-j}} \dif \mult z\\
		&\qquad \sum_{\substack{J_1\subseteq \set{1,…,j}\\J_2\subseteq \set{1,…,l-j}}} (-1)^{l-\abs{J_1}-\abs{J_2}} \tilde{ψ}_{\mult k}^! \biggl(\mult x_I;\sum_{i∈J_1}δ_{y_i}\biggr) \tilde{ψ}_{\mult k}^! \biggl(\mult x_{I^c};\sum_{i∈J_2}δ_{z_i}\biggr)\\
		&\qquad \Bigl( ρ^{(\abs{I}+p)}(\mult x,\mult y,\mult z) - ρ^{(\abs{I}+j)}(\mult x_I,\mult y)ρ^{(\abs{I^c}+l-j)}(\mult x_{I^c},\mult z) \Bigr) \Biggr|\\
		&\le \sum_{l=0}^∞\sum_{j=0}^l \frac{1}{j!(l-j)!} 2^l C_1^{\abs{\mult k}} \max\set[\big]{1,t}^{γ_1\abs{\mult k}} \bigl(j+\abs{I}\bigr)^{γ_2\abs{\mult k_I}} \bigl(l-j+\abs{I^c}\bigr)^{γ_2\abs{\mult k_{I^c}}}\\
		&\qquad C_2^{p+l} (p+l)!^a Φ\bigl(\dist(\mult x_I,\mult x_{I^c})-2t\bigr) \bigl(\abs{I}ϑ_dt^d\bigr)^j \bigl(\abs{I^c}ϑ_dt^d\bigr)^{l-j}\\
		&\le C_1^{\abs{\mult k}} C_2^p \max\set[\big]{1,t}^{γ_1\abs{\mult k}} Φ\bigl(\dist(\mult x_I,\mult x_{I^c})-2t\bigr)*\\
		&\qquad \sum_{l=0}^∞ \sum_{j=0}^l \frac{(2C_2)^l (p+l)!^a (p+l)^{γ_2\abs{\mult k}}}{j!(l-j)!} \bigl(\abs{I}ϑ_dt^d\bigr)^j \bigl(\abs{I^c}ϑ_dt^d\bigr)^{l-j}\\
		&\le C_1^{\abs{\mult k}} C_2^p \max\set[\big]{1,t}^{γ_1\abs{\mult k}} Φ\bigl(\dist(\mult x_I,\mult x_{I^c})-2t\bigr) \sum_{l=0}^∞ \frac{(2C_2ϑ_dpt^d)^l (p+l)!^a (p+l)^{γ_2\abs{\mult k}}}{l!}.
	\end{align*}
	Apply now \cref{lem:TouchardSum} to estimate the sum. Use $(p+l)!^a \le 2^{a(p+l)} p!^a l!^a$ and $(p+l)^{γ_2\abs{\mult k}} \le 2^{γ_2\abs{\mult k}} \bigl(p^{γ_2\abs{\mult k}} + l^{γ_2\abs{\mult k}}\bigr)$ to bound the sum as follows:
	\begin{align*}
		\MoveEqLeft 2^{ap+γ_2\abs{\mult k}} \sum_{l=0}^∞ \frac{p!^a l!^a (p^{γ_2\abs{\mult k}} + l^{γ_2\abs{\mult k}})}{l!} \bigl(2^{1+a}C_2ϑ_dpt^d\bigr)^l\\
		&\le 2*2^{ap+γ_2\abs{\mult k}} p!^a \biggl( \frac{ep^{γ_2\abs{\mult k}}}{(1-a)} e^{(4C_2ϑ_dpt^d)^\frac{1}{1-a}} + \frac{e^{1+\ceil{γ_2\abs{\mult k}}}(1+\ceil{γ_2\abs{\mult k}})!}{(1-a)^{1+\ceil{γ_2\abs{\mult k}}}} e^{(4C_2ϑ_dpt^d)^\frac{1}{1-a}} \biggr).
	\end{align*}
	Next, observe that $\ceil{1+γ_2\abs{\mult k}}! \le C_3^{\abs{\mult k}} \abs{\mult k}!^{γ_2}$ and $p^{γ_2\abs{\mult k}} \le \abs{\mult k}!^{γ_2}$ for some constant $C_3\ge1$. By combining the estimate for the sum with the estimate above, we therefore proved the existence of a constant $C\ge1$ (independent of $p$, $\mult k$, $\mult x$, $n$ and $I$) such that
	\begin{align*}
		\MoveEqLeft \abs[\big]{\tilde{m}_{\mult k}(\mult x;n) - \tilde{m}_{\mult k_I}(\mult x_I;n) \tilde{m}_{\mult k_{I^c}}(\mult x_{I^c};n)}\\
		&\le C^{\abs{\mult k}} p!^a \abs{\mult k}!^{γ_2} \max\set[\big]{1,t}^{γ_1\abs{\mult k}} Φ\bigl(\dist(\mult x_I,\mult x_{I^c})-2t\bigr) e^{Cp^\frac{1}{1-a}t^\frac{d}{1-a}}.
	\end{align*}
	To be able to conclude, we need to distinguish between the two cases $t\le1$ and $t\ge1$. In the first one the statement is already proven. For the latter one use $t\le e^{t^d}$ and hence $t^{γ_1\abs{\mult k}} \le e^{γ_1\abs{\mult k}t^d} \le e^{γ_1\abs{\mult k}^\frac{1}{1-a}t^\frac{d}{1-a}}$ to finish the proof.
\end{proof}
We are now ready to prove the explicit fast decay of correlation functions for the $ξ$-weighted measure.
\begin{proposition}\label{thm:FastDecayXiWeightedMeasure}
	Assume that the point process $\cP$ satisfies $\EDC(a,\hat{a})$ and that the score function $ξ$ satisfies $\ST(b)$, $\MG(β)$ and $\PG(γ_1,γ_2)$. Then, there exist constants $c>0$ and $C\ge1$ such that for all $p∈ℕ$, $\mult k∈ℕ^p$, distinct $\mult x∈(ℝ^d)^p$, $n∈ℕ$, $\emptyset \ne I\subsetneq \set{1,…,p}$ and $N∈ℕ$ it holds
	\begin{align*}
		\MoveEqLeft \abs[\big]{m_{\mult k}(\mult x;n) - m_{\mult k_I}(\mult x_I;n)m_{\mult k_{I^c}}(\mult x_{I^c};n)}\\
		&\le C^{\abs{\mult k} + N} p!^α \abs{\mult k}^{N\max\set{\frac{1}{\hat{a}}+\frac{a}{d},\frac{1}{d}}} \abs{\mult k}!^β N!^b \Bigl(\max\set[\Big]{\frac{\dist(\mult x_I,\mult x_{I^c})}{3},1}\Bigr)^{-N\min\set{1,\frac{(1-a)\hat{a}}{d}}}\\
		&\qquad + C^{\abs{\mult k}} p!^a \abs{\mult k}!^{γ_2} e^{-c(\max\set{\frac{\dist(\mult x_I,\mult x_{I^c})}{3},1})^{\hat{a}}}.
	\end{align*}
\end{proposition}
\begin{proof}
	If $\dist(\mult x_I,\mult x_{I^c}) \le 3$, apply \cref{eq:BddXiWeightedMeasure} to obtain
	\begin{equation*}
		\abs[\big]{m_{\mult k}(\mult x;n) - m_{\mult k_I}(\mult x_I;n)m_{\mult k_{I^c}}(\mult x_{I^c};n)}
		\le C^{\abs{\mult k}} p!^α \abs{\mult k}!^β
	\end{equation*}
	for some constant $C\ge 1$.
	
	From now on, assume that $\dist(\mult x_I,\mult x_{I^c}) \ge 3$. By applying \cref{lem:BoundM,lem:BoundMTilde}, there exists a constant $C_1\ge1$ such that
	\begin{align*}
		\MoveEqLeft \abs[\big]{m_{\mult k}(\mult x;n) - m_{\mult k_I}(\mult x_I;n)m_{\mult k_{I^c}}(\mult x_{I^c};n)}\\
		&\le C_1^{\abs{\mult k} + N} p!^α \abs{\mult k}!^β \frac{N!^b}{t^N} + C_1^{\abs{\mult k}} p!^a \abs{\mult k}!^{γ_2} Φ\bigl(\dist(\mult x_I,\mult x_{I^c})-2t\bigr) e^{C_1\abs{\mult k}^\frac{1}{1-a}t^\frac{d}{1-a}}.
	\end{align*}
	Observe that $\frac{(1-a)\hat{a}}{d} \ge 1$ if and only if $\frac{1}{d} \ge \frac{1}{\hat{a}} + \frac{a}{d}$. Hence, we can prove the statement by considering the two different cases $(1-a)\hat{a}\le d$ and $(1-a)\hat{a}\ge d$.
	
	Assume first that $(1-a)\hat{a}\le d$. As $Φ$ is exponentially fast decaying, there exist constants $c$ and $C_2$ such that 
	\begin{equation*}
		Φ(s) \le C_2 e^{-2cs^{\hat{a}}}
	\end{equation*}
	for all $s∈ℝ_+$. Without loss of generality we take $c \le C_1$. Define the parameter $t$ by $t = \bigl(\frac{c}{C_1}\bigr)^{\frac{1-a}{d}} \abs{\mult k}^{-\frac{1}{d}}\bigl(\frac{\dist(\mult x_I,\mult x_{I^c})}{3}\bigr)^{\frac{(1-a)\hat{a}}{d}}$. In particular, as $(1-a)\hat{a}\le d$, one obtains the bound $t \le \frac{\dist(\mult x_I,\mult x_{I^c})}{3}$. Thus, $Φ\bigl(\dist(\mult x_I,\mult x_{I^c})-2t\bigr) \le Φ\bigl(\frac{\dist(\mult x_I,\mult x_{I^c})}{3}\bigr)$ and
	\begin{equation*}
		Φ\bigl(\dist(\mult x_I,\mult x_{I^c})-2t\bigr) e^{C_1\abs{\mult k}^\frac{1}{1-a}t^\frac{d}{1-a}}
		\le C_2e^{-c\bigl(\frac{\dist(\mult x_I,\mult x_{I^c})}{3}\bigr)^{\hat{a}}}.
	\end{equation*}
	Evaluating the above-mentioned bound at our choice of $t$ therefore yields
	\begin{align*}
		\MoveEqLeft \abs[\big]{m_{\mult k}(\mult x;n) - m_{\mult k_I}(\mult x_I;n)m_{\mult k_{I^c}}(\mult x_{I^c};n)}\\
		&\le C_1^{\abs{\mult k} + N} \biggl(\frac{c}{C_1}\biggr)^{\frac{(1-a)N}{d}} p!^α \abs{\mult k}^\frac{N}{d} \abs{\mult k}!^β N!^b \biggl(\frac{\dist(\mult x_I,\mult x_{I^c})}{3}\biggr)^{-N\frac{(1-a)\hat{a}}{d}}\\
		&\qquad + C_2C_1^{\abs{\mult k}} p!^a \abs{\mult k}!^{γ_2} e^{-c\bigl(\frac{\dist(\mult x_I,\mult x_{I^c})}{3}\bigr)^{\hat{a}}}.
	\end{align*}
	This proves the claim in case of $(1-a)\hat{a}\le d$.
	
	Now, consider the second case where $(1-a)\hat{a} \ge d$. Choose $t = \abs{\mult k}^{-\frac{1}{\hat{a}}-\frac{a}{d}}\bigl(\frac{\dist(\mult x_I,\mult x_{I^c})}{3}\bigr)$. Clearly it holds that $t\le \frac{\dist(\mult x_I,\mult x_{I^c})}{3}$, and hence $Φ\bigl(\dist(\mult x_I,\mult x_{I^c})-2t\bigr) \le Φ\bigl(\frac{\dist(\mult x_I,\mult x_{I^c})}{3}\bigr) \le C_2e^{-2c\bigl(\frac{\dist(\mult x_I,\mult x_{I^c})}{3}\bigr)^{\hat{a}}}$. Denote the maximum of $s\mapsto -cs^{\hat{a}} + C_1\abs{\mult k}^{\frac{1}{1-a}}\abs{\mult k}^{-\frac{d}{(1-a)\hat{a}} -\frac{a}{1-a}} s^\frac{d}{1-a}$ within $ℝ_+$ by $s_0$. The maximum is unique and it turns out that one can bound $-cs_0^{\hat{a}} + C_1\abs{\mult k}^{\frac{1}{1-a}}\abs{\mult k}^{-\frac{d}{(1-a)\hat{a}} -\frac{a}{1-a}} s_0^\frac{d}{1-a} \le C_3 \abs{\mult k}$ for some constant $C_3>1$. Hence, we conclude that in this case
	\begin{align*}
		\MoveEqLeft \abs[\big]{m_{\mult k}(\mult x;n) - m_{\mult k_I}(\mult x_I;n)m_{\mult k_{I^c}}(\mult x_{I^c};n)}\\
		&\le C_1^{\abs{\mult k} + N} p!^α \abs{\mult k}!^β \frac{N!^b}{t^N} + C_1^{\abs{\mult k}} p!^a \abs{\mult k}!^{γ_2} Φ\bigl(\dist(\mult x_I,\mult x_{I^c})-2t\bigr) e^{C_1\abs{\mult k}^\frac{1}{1-a}t^\frac{d}{1-a}}\\
		&\le C^{\abs{\mult k} + N} p!^α \abs{\mult k}!^β \abs{\mult k}^{\frac{N}{\hat{a}}+\frac{aN}{d}} N!^b \biggl(\frac{\dist(\mult x_I,\mult x_{I^c})}{3}\biggr)^{-N}\\
		&\qquad + C^{\abs{\mult k}} p!^a \abs{\mult k}!^{γ_2} e^{-c\bigl(\frac{\dist(\mult x_I,\mult x_{I^c})}{3}\bigr)^{\hat{a}}}
	\end{align*}
	for $C$ chosen sufficiently large. This finishes the proof.
\end{proof}


\subsection{Proof of the deviation results}

We now exploit the fast decay of correlations for the $ξ$-weighted measure $μ_n^ξ$ proven in \cref{thm:FastDecayXiWeightedMeasure} to conclude an explicit bound on cumulants.

Let us first introduce some further notation: By $\cQ(I)$ we denote the set of all set partitions of $I\subseteq ℕ$. In case of $I=\set{1,…,p}$ for some $p∈ℕ$, we also write $\cQ_p$ for short. Set partitions will be usually denoted by capital Greek letters $Π,Σ,…$ and their parts by the corresponding small Greek letters $π,σ,…$, possibly with indices. Sometimes we need to work with ordered set partitions, in particular in \cref{lem:ClusteringCumu}. To do so, it is necessary that we always assume the parts of such a partition to be ordered according to their smallest element. Moreover, we use $\prec$ to denote refinements of set partitions. Finally, we denote by $\cS_p$ the symmetric group on $\set{1,…,p}$.

We are now ready to state the general clustering result for cumulants. A similar but less explicit result can be found in \cite{BESY08MDPGeometricProbability,ERS15MDPForStabilizingFunctionals}, for instance.
\begin{lemma}\label{lem:ClusteringCumu}
	Consider a family of numbers (moments) $(m_I)_{I\subseteq \set{1,…,p}}$ and define the corresponding cumulants by
	\begin{equation*}
		κ_I = \sum_{Π∈\cQ(I)} (-1)^{\abs{Π}-1}\bigl(\abs{Π}-1\bigr)!\prod_{π∈Π} m_π.
	\end{equation*}
	For disjoint sets $I,J\subseteq \set{1,…,p}$, denote by $δ_{I,J} = m_{I\cup J} - m_I m_J$ the moment clusters. Then, for any $\emptyset\ne I\subsetneq\set{1,…,p}$ with $1∈I$ it holds
	\begin{equation*}
		κ_{\set{1,…,p}} = \sum_{Π\prec \set{I,I^c}} \sum_{\substack{τ∈\cS_{\abs{Π}}\\τ(1)=1}} (-1)^{\abs{Π}+\abs{D(Π,τ)}-1} \prod_{(σ_1,σ_2)∈D(Π,τ)} δ_{σ_1,σ_2} \prod_{σ∈M(Π,τ)} m_σ,
	\end{equation*}
	where $Π = \set{π_i}_{i=1,…,\abs{Π}}$ with the $π_i$ ordered according to their smallest element and 
	\begin{align*}
		D(Π,τ) &= \set[\Big]{(σ_1,σ_2) ∈Π^2 \given \exists i∈\set{1,…,p-1}\colon σ_1=π_{τ(i)}\subseteq I\text{ and }σ_2=π_{τ(i+1)}\subseteq I^c},\\
		M(Π,τ) &= \set[\Big]{σ∈Π \given \nexists \tilde{σ}∈Π \colon (\tilde{σ},σ)∈D(Π,τ) \text{ or }(σ,\tilde{σ})∈D(Π,τ)}.
	\end{align*}
	Moreover, $D(Π,τ)$ is non-empty and $M(Π,τ) \uplus D(Π,τ)_1 \uplus D(Π,τ)_2 = Π$, where $D(Π,τ)_i$ denotes the projection on the $i$-th coordinate and $\uplus$ stands for disjoint union.
\end{lemma}
\begin{remark}
	The condition $1∈I$ can always be enforced by considering $I^c$ instead of $I$ if necessary. The two sums in the prior lemma can also be viewed as ordered set partitions (set compositions) with each part being contained in either $I$ or $I^c$. Any summand can be thought of as corresponding to one ordered set partition with a cluster term appearing whenever two consecutive parts belong to $I$ and $I^c$, respectively. If not, the part appears as a moment term.
\end{remark}
\begin{remark}
	Clustering the terms appearing within the product over $M(Π,τ)$ further might enable us to reduce the parameter $γ$ in the bound on cumulants in \cref{thm:BoundCumulants}. For example, for $n=5$ and $I=\set{1,2,3}$ we might exploit the further clustering of the two summands $δ_{3,45}m_{1}m_{2}$ and $-δ_{3,45}m_{12}$ in order to obtain the single summand $-δ_{1,2}δ_{3,45}$. Details may appear in a forthcoming work.
\end{remark}
\begin{proof}
	Note that
	\begin{equation*}
		κ_{\set{1,…,p}} = \sum_{Π∈\cQ_p} \sum_{\substack{τ∈\cS_{\abs{Π}}\\τ(1)=1}} (-1)^{\abs{Π}-1} \prod_{π∈Π} m_π.
	\end{equation*}
	Think of a pair $(Π,τ)$ as an ordered set partition $(π_1,…,π_k)$ for some $k∈ℕ$ such that $1∈π_1$. To any such ordered set partition associate its $I$-refinement defined by $(π_1\cap I,π_1\cap I^c,…,π_k\cap I,π_k\cap I^c)$. Observe that any ordered set partition has a unique $I$-refinement which is given by pairs $(Π,τ)$ with $Π\prec \set{I,I^c}$, $τ∈\cS_{\abs{Π}}$, $τ(1)=1$. Hence,
	\begin{equation*}
		κ_{\set{1,…,p}}
		= \sum_{Π\prec\set{I,I^c}} \sum_{\substack{τ∈\cS_{\abs{Π}}\\τ(1)=1}} \sum_{\substack{Σ∈\cQ_n, ρ∈\cS_{\abs{Σ}}\\(Π,τ)\text{ is $I$-refinement of }(Σ,ρ)}} (-1)^{\abs{Σ}-1} \prod_{σ∈Σ} m_σ.
	\end{equation*}
	Given any such pair $(Π,τ)$, $Π\prec\set{I,I^c}$ and $τ∈\cS_{\abs{Π}}$ with $τ(1)=1$, there are exactly $2^l$ ordered set partitions $(Σ,ρ)$, $Σ∈\cQ_p$ and $ρ∈\cS_{\abs{Σ}}$ with $ρ(1)=1$, with $I$-refinement $(Π,τ)$, where $l=\abs{D(Π,τ)}$ denotes the number of consecutive parts in $I$ and $I^c$. These $2^l$ summands can be grouped to obtain
	\begin{equation*}
		\qquad \sum_{\substack{Σ∈\cQ_n, ρ∈\cS_{\abs{Σ}}\\\mathclap{(Π,τ)\text{ is $I$-refinement of }(Σ,ρ)}}} (-1)^{\abs{Σ}-1} \prod_{σ∈Σ} m_σ
		= (-1)^{\abs{Π}+\abs{D(Π,τ)}-1} \prod_{(σ_1,σ_2)∈D(Π,τ)} δ_{σ_1,σ_2} \prod_{σ∈M(Π,τ)} m_σ.
	\end{equation*}
	The proof is hereby concluded.
\end{proof}
\begin{lemma}\label{lem:SumSetPartition}
	For any $p∈ℕ$ and any $c∈ℝ_+$ it holds that
	\begin{equation*}
		\sum_{Π∈\cQ_p} \abs{Π}!^c \prod_{π∈Π} \abs{π}!^c
		\le 2^p p!^{\max\set{1,c}}.
	\end{equation*}
	The exponent $\max\set{1,c}$ cannot be reduced.
\end{lemma}
\begin{proof}
	Compare \cite[Lemma 3.5]{ERS15MDPForStabilizingFunctionals}.
\end{proof}
In analogy to standard cumulants, let us introduce the factorial cumulant density (sometimes also called Ursell function, truncated correlation functions or connected correlation function):
\begin{equation*}
	κ_{\mult k}(\mult x;n) = \sum_{Π∈\cQ_p} (-1)^{\abs{Π}-1}\bigl(\abs{Π}-1\bigr)!\prod_{π∈Π} m_{\mult k_π}(\mult x_π;n).
\end{equation*}
Sometimes, it is also denoted by $m_{\mult k}^\top(\mult x;n)$ instead.
\begin{lemma}\label{lem:BoundFactorialCumulantDensity}
	Assume that the point process $\cP$ satisfies $\EDC(a,\hat{a})$ and that the score function $ξ$ satisfies $\ST(b)$, $\MG(β)$ and $\PG(γ_1,γ_2)$. Then, for every $p∈ℕ$, $\mult k∈ℕ^p$, distinct $\mult x∈(ℝ^d)^p$, $n∈ℕ$, $\emptyset\ne I\subsetneq \set{1,…,p}$ and $N∈ℕ$ it holds
	\begin{align*}
		\abs[\big]{κ_{\mult k}(\mult x;n)}
		&\le C^{\abs{\mult k}} p! \abs{\mult k}!^β C^N \abs{\mult k}^\frac{N}{d} N!^b \biggl(\max\set[\bigg]{\frac{\dist(\mult x_I,\mult x_{I^c})}{3},1}\biggr)^{-N\frac{(1-a)\hat{a}}{d}}\\
		&\qquad +  C^{\abs{\mult k}} p! \abs{\mult k}!^{\max\set{γ_2,β}} e^{-c\max\set[\big]{\frac{\dist(\mult x_I,\mult x_{I^c})}{3},1}^{\hat{a}}}.
	\end{align*}
\end{lemma}
\begin{proof}
	The idea is to apply the clustering result from \cref{lem:ClusteringCumu}. To be able to do this, define for any two disjoint index sets $\emptyset\ne J_1,J_2\subsetneq\set{1,…,p}$ the moment cluster
	\begin{equation*}
		δ_{\mult k_{J_1},\mult k_{J_2}}\bigl(\mult x_{J_1},\mult x_{J_2};n\bigr)
		= m_{\mult k_{J_1\cup J_2}} \bigl(\mult x_{J_1\cup J_2};n\bigr) - m_{\mult k_{J_1}} \bigl(\mult x_{J_1};n\bigr) m_{\mult k_{J_2}} \bigl(\mult x_{J_2};n\bigr).
	\end{equation*}
	Thus, \cref{thm:FastDecayXiWeightedMeasure} together with $(ν_1+ν_2)! \le 2^{ν_1+ν_2} ν_1! ν_2!$ for all $ν_1,ν_2∈ℕ$ implies that
	\begin{align*}
		\abs[\Big]{δ_{\mult k_{J_1},\mult k_{J_2}} \bigl(\mult x_{J_1},\mult x_{J_2};n\bigr)}
		&\le (4C_1)^{\abs{\mult k_{J_1\cup J_2}} + N} \abs{J_1}!^α \abs{J_2}!^α \abs{\mult k_{J_1\cup J_2}}^{N\max\set[\big]{\frac{1}{\hat{a}}+\frac{a}{d},\frac{1}{d}}} \abs{\mult k_{J_1}}!^β \abs{\mult k_{J_2}}!^β N!^b\\
		&\qquad\qquad * \biggl(\max\set[\bigg]{\frac{\dist(\mult x_{J_1},\mult x_{J_2})}{3},1}\biggr)^{-N\min\set[\big]{1,\frac{(1-a)\hat{a}}{d}}}\\
		&\quad + (4C_1)^{\abs{\mult k_{J_1\cup J_2}}} \abs{J_1}!^a \abs{J_2}!^a \abs{\mult k_{J_1}}!^{γ_2} \abs{\mult k_{J_2}}!^{γ_2} e^{-c\bigl(\max\set[\big]{\frac{\dist(\mult x_{J_1},\mult x_{J_2})}{3},1}\bigr)^{\hat{a}}}
	\end{align*}
	for some constant $C_1\ge1$. Now, fix $\emptyset\ne I\subsetneq\set{1,…,p}$. Then, for any $\emptyset\ne J_1\subseteq I$ and $\emptyset\ne J_2\subseteq I^c$ the bound
	\begin{equation*}
		\dist\bigl(\mult x_{J_1},\mult x_{J_2}\bigr) \ge \dist\bigl(\mult x_I,\mult x_{I^c}\bigr)
	\end{equation*}
	holds. Since $s\mapsto \max\set{\frac{s}{3},1}^{-N\frac{(1-a)\hat{a}}{d}}$ and $s\mapsto e^{-c\max\set{s,1}^{\hat{a}}}$ are both decreasing, one immediately obtain
	\begin{equation*}
		\biggl(\max\set[\bigg]{\frac{\dist(\mult x_{J_1},\mult x_{J_2})}{3},1}\biggr)^{-N\min\set[\big]{1,\frac{(1-a)\hat{a}}{d}}}
		\le \biggl(\max\set[\bigg]{\frac{\dist(\mult x_I,\mult x_{I^c})}{3},1}\biggr)^{-N\min\set[\big]{1,\frac{(1-a)\hat{a}}{d}}}
	\end{equation*}
	and
	\begin{equation*}
		e^{-c\max\set[\big]{\frac{\dist(\mult x_{J_1},\mult x_{J_2})}{3},1}^{\hat{a}}}
		\le e^{-c\max\set[\big]{\frac{\dist(\mult x_I,\mult x_{I^c})}{3},1}^{\hat{a}}}.
	\end{equation*}
	Recall from \cref{eq:BddXiWeightedMeasure} that for any $\emptyset\ne J\subsetneq\set{1,…,p}$
	\begin{equation*}
		\abs[\big]{m_{\mult k_J}\bigl(\mult x_J;n\bigr)}
		\le C^{\abs{\mult k_J}} \abs{J}!^α \abs{\mult k_J}!^β,
	\end{equation*}
	which implies the estimate
	\begin{equation*}
		\abs[\big]{δ_{\mult k_{J_1},\mult k_{J_2}}\bigl(\mult x_{J_1},\mult x_{J_2};n\bigr)}
		\le 2 (4C_1)^{\abs{\mult k_{J_1\cup J_2}}} \abs{J_1}!^α \abs{J_2}!^α \abs{\mult k_{J_1}}!^β \abs{\mult k_{J_2}}!^β.
	\end{equation*}
	We are now ready to apply \cref{lem:ClusteringCumu} with the result of obtaining	\begin{align*}
		\MoveEqLeft[1] \abs[\big]{κ_{\mult k}(\mult x;n)}
		= \abs[\Bigg]{\sum_{Π\prec \set{I,I^c}} \sum_{\substack{τ∈\cS_{\abs{Π}}\\τ(1)=1}} (-1)^{\abs{Π}+\abs{D(Π,τ)}-1} \smashoperator{\prod_{(σ_1,σ_2)∈D(Π,τ)}} δ_{\mult k_{σ_1},\mult k_{σ_2}}\bigl(\mult x_{σ_1},\mult x_{σ_2};n\bigr) \smashoperator{\prod_{σ∈M(Π,τ)}} m_{\mult k_σ}\bigl(\mult x_σ;n\bigr)}\\
		&\le 2^p (4C_1)^{\abs{\mult k}} \sum_{Π\prec \set{I,I^c}} \sum_{\substack{τ∈\cS_{\abs{Π}}\\τ(1)=1}}
		\Biggl( \prod_{π∈Π} \abs{π}!^a \abs{\mult k_π}!^{\max\set{γ_2,β}} e^{-c\max\set[\big]{\frac{\dist(\mult x_I,\mult x_{I^c})}{3},1}^{\hat{a}}} \\
		&\quad + \prod_{π∈\Pi} \abs{π}!^α \abs{\mult k_π}!^β C_1^N \abs{\mult k_π}^{N\max\set[\big]{\frac{1}{\hat{a}}+\frac{a}{d},\frac{1}{d}}} N!^b \biggl(\max\set[\bigg]{\frac{\dist(\mult x_I,\mult x_{I^c})}{3},1}\biggr)^{-N\min\set[\big]{1,\frac{(1-a)\hat{a}}{d}}} \Biggr).
	\end{align*}
	Finally, \cref{lem:SumSetPartition} yields
	\begin{align*}
		\MoveEqLeft[1] \abs[\big]{κ_{\mult k}(\mult x;n)}\\
		&\le 4^p (4C_1)^{\abs{\mult k}}
		\biggl( p!^{\max\set{1,a}} \abs{\mult k}!^{\max\set{γ_2,β}} e^{-c\max\set[\big]{\frac{\dist(\mult x_I,\mult x_{I^c})}{3},1}^{\hat{a}}}\\
		&\quad + p!^{\max\set{1,α}} \abs{\mult k}!^β C_1^N \abs{\mult k}^{N\max\set[\big]{\frac{1}{\hat{a}}+\frac{a}{d},\frac{1}{d}}} N!^b \biggl(\max\set[\bigg]{\frac{\dist(\mult x_I,\mult x_{I^c})}{3},1}\biggr)^{-N\min\set[\big]{1,\frac{(1-a)\hat{a}}{d}}} \biggr).
	\end{align*}
	Conclude using $α\le a\le 1$.
\end{proof}
In the following, it is necessary to calculate integrals of functions depending on $\dist(\mult x_I,\mult x_{I^c})$. To obtain our more precise cumulant bound, we use the following version of the coarea formula instead of drawing upon elementary bounds.
\begin{lemma}\label{lem:Coarea}
	Consider $f\colon ℝ_+ \to ℝ_+$ and $u\colon ℝ^d \to ℝ_+$ such that $u$ is Lipschitz continuous and homogeneous, i.e.\ $u(s\mult x) = \abs{s}u(\mult x)$ for all $s∈ℝ$ and $\mult x∈ℝ^d$. Then,
	\begin{equation*}
		\int_{ℝ^d} f\bigl(u(\mult x)\bigr) \dif \mult x
		= d\Vol_{d}(u\le1) \int_{ℝ_+} s^{d-1}f(s) \dif s.
	\end{equation*}
\end{lemma}
\begin{proof}
	Recall the coarea formula: For any Lipschitz continuous function $u\colon ℝ^d \to ℝ_+$ and any function $f\colon ℝ_+ \to ℝ_+$ it holds
	\begin{equation*}
		\int_{ℝ^d} f\bigl(u(x)\bigr) \dif \mult x = \int_{ℝ_+} \int_{u(\mult x)=s} \frac{f(s)}{\norm{\nabla u(\mult x)}} \dif \mult x \dif s,
	\end{equation*}
	where the inner integration on the right-hand side is with respect to the $d-1$-dimensional Lebesgue measure restricted to the set $\set{\mult x∈ℝ^d \given u(\mult x)=s}$. We now apply this formula to our case above. Notice that
	\begin{equation*}
		\abs{s}\nabla_{\mult x}u(\mult x)\big\vert_{\mult x=s\mult y}
		= \nabla_{\mult y} u(t\mult y)
		= \abs{s} \nabla_{\mult y} u(\mult y),
	\end{equation*}
	since $u$ is homogeneous. Thus,
	\begin{align*}
		\int_{ℝ^d} f\bigl(u(x)\bigr) \dif \mult x
		&= \int_{ℝ_+} \int_{u(\mult x)=s} \frac{f(s)}{\norm{\nabla u(\mult x)}} \dif \mult x \dif s
		= \int_{ℝ_+} \int_{u(s\mult y)=s} \frac{s^{d-1}f(s)}{\norm[\big]{\nabla u(\mult x)\big\vert_{\mult x=s\mult y}}} \dif \mult y \dif s\\
		&= \int_{ℝ_+} s^{d-1}f(s) \dif s \int_{u(\mult y)=1} \frac{1}{\norm[\big]{\nabla_{\mult y}u(\mult y)}} \dif \mult y.
	\end{align*}
	Next, apply this equation to the function $f(s) = \ind_{\set{s\le1}}$ to conclude
	\begin{equation*}
		\Vol_{d}\bigl(u\le1\bigr) = \frac{1}{d}\int_{u(\mult y)=1} \frac{1}{\norm[\big]{\nabla_{\mult y}u(\mult y)}} \dif \mult y.
	\end{equation*}
	The claim now follows by combining both equations.
\end{proof}
Let us introduce the norm $\norm{\cdot}_{sig}$ on $(ℝ^d)^{p-1}$:
\begin{equation*}
	\norm{\mult x}_{sig} = \max_{I\subseteq\set{1,…,p-1}}\dist\bigl((0,\mult x_I),\mult x_{I^c}\bigr).
\end{equation*}
The index “$\text{sig}$” stands for “sphere of influence graph” as the norm recognizes if the sphere of influence graph of $(0,\mult x)$ is connected. The precise statement is part of the proof of the following volume bound:
\begin{lemma}\label{lem:VolBound}
	The following volume bound holds:
	\begin{equation*}
		\Vol_{d(p-1)}\bigl(\norm{\cdot}_{sig}\le1\bigr) \le (eϑ_d)^{p-1} p!.
	\end{equation*}
\end{lemma}
\begin{proof}
	For reasons of simplicity, denote $x_0=0$. Consider the sphere of influence graph of radius $r∈ℝ_+$ associated to $x_0,…,x_{p-1}$: Its vertex set is $\set{0,…,p-1}$, and there is an edge between $i$ and $j$ if and only if $x_i$ and $x_j$ are at distance at most $r$. Denote this graph for short by $\SIG_r(0,\mult x)$. We first show that $\norm{\mult x}_{sig} \le r$ if and only if $\SIG_r(0,\mult x)$ is connected. Indeed, this can be easily seen when considering the contraposition. If $\SIG_r(0,\mult x)$ is not connected, there is a set $I\subsetneq\set{1,…,p-1}$ such that $\set{0}\cup I$ and $I^c$ are not connected. Hence, $\norm{\mult x}_{sig} \ge \dist((0,\mult x_I),\mult x_{I^c}) > r$. For the other direction, assume that $\norm{\mult x}>r$. Then, there is a set $I$ such that $\dist((0,\mult x_I),\mult x_{I^c}) > r$. So clearly, $\set{0}\cup I$ and $I^c$ cannot be connected in $\SIG_r(0,\mult x)$.
	
	Let us now calculate $\Vol_{d(p-1)}(\norm{\cdot}_{sig}\le1)$. Since $\SIG_1(0,\mult x)$ is connected for any $\mult x$ with $\norm{\mult x}_{sig} \le 1$, we can find a minimal spanning tree of the graph. Let us denote its edges by $\set{(s(i),e(i))\given i∈\set{1,…,p-1}}$, where $s$ and $e$ are permutations of $\set{0,…,p-1}$ with $s(0)=0$ and $e(i)∈\set{s(j)\given j<i}$, i.e.\ each new vertex $s(i)$ is connected to one of the already discovered vertices $s(1),…,s(i-1)$. Then,
	\begin{align*}
		\MoveEqLeft \Vol_{d(p-1)}\bigl(\norm{\mult x}_{sig} \le 1\bigr)\\
		&\le \sum_{\text{labelled tree } t} \int_{\set{\mult x\given \SIG_1(0,\mult x)\text{ has minimal spanning tree } t}} \dif \mult x\\
		&= \sum_{\text{labelled tree } t} \prod_{i=1}^p \int_{\norm{x_{s(i)}-x_{e(i)}}\le1} \dif x_{e(i)}\\
		&= \sum_{\text{labelled tree } t} ϑ_d^{p-1}\\
		&= ϑ_d^{p-1} p^{p-2}.
	\end{align*}
	This concludes the proof.
\end{proof}
\begin{remark}
	The exponent of $p!$ cannot be reduced.
\end{remark}
\begin{corollary}\label{cor:IntegrateTreeDist}
	For any $p,l∈ℕ$ with $l>d(p-1)+1$, there exists a constant $C\ge1$ such that
	\begin{equation*}
		\int_{(ℝ^d)^{p-1}} \bigl(\max\set{\norm{\mult x}_{sig},1}\bigr)^{-l} \dif \mult x
		\le C^p p!.
	\end{equation*}
	For any $c,\hat{a}>0$, there exists a constant $C\ge1$ such that
	\begin{equation*}
		\int_{(ℝ^d)^{p-1}} e^{-c\max\set{\norm{\mult x}_{sig},1}^{\hat{a}}} \dif \mult x
		\le C^p p!^{1+\frac{d}{\hat{a}}}.
	\end{equation*}
\end{corollary}
\begin{proof}
	Combine \cref{lem:Coarea,lem:VolBound} with
	\begin{equation*}
		\int_{ℝ_+} s^{d(p-1)}\max\set{s,1}^{-l} \dif s
		= \frac{1}{d(p-1)+1} + \frac{1}{l-d(p-1)-1}
		\le \frac{2}{d}
	\end{equation*}
	and
	\begin{equation*}
		\int_{ℝ_+} s^{d(p-1)} e^{-c\max\set{s,1}^{\hat{a}}} \dif s
		\le \frac{1}{d(p-1)+1} + \frac{c^{-\frac{1+d(p-1)}{\hat{a}}}}{\hat{a}} Γ(\frac{1+d(p-1)}{\hat{a}})
	\end{equation*}
	to obtain the result.
\end{proof}
\begin{theorem}\label{thm:BoundCumulants}
	Assume that the point process $\cP$ satisfies $\EDC(a,\hat{a})$ and that the score function $ξ$ satisfies $\ST(b)$, $\MG(β)$ and $\PG(γ_1,γ_2)$. Then, there exists a constant $C\ge1$ such that for any $k∈ℕ$, $n∈ℕ$ and any bounded function $f \colon ℝ^d \to ℝ$ the following bound on cumulants holds:
	\begin{equation*}
		\abs[\big]{κ^{(k)}\bigl(μ_n^ξ(f)\bigr)} \le
		\begin{cases} n\norm{f}_∞^k C^k k!^{2+\max\set{γ_2,β} + \frac{d}{(1-a)\hat{a}} + \frac{bd^2}{(1-a)\hat{a}}} & \text{if }\frac{(1-a)\hat{a}}{d}\le 1,\\
		n\norm{f}_∞^k C^k k!^{2+\max\set{γ_2,β} + \frac{d}{\hat{a}} + a + bd} & \text{if }\frac{(1-a)\hat{a}}{d}\ge 1.
	\end{cases}
	\end{equation*}
\end{theorem}
\begin{proof}
	Throughout the proof there are set partitions $Π∈\cQ_k$ appearing. By convention, we denote the parts of $Π$ by $π_1,…,π_p$, where $p=\abs{Π}$. By using the refined Campbell-Mecke formula from \cref{eq:refinedCampbell}, the moments of $μ_n^ξ(f)$ are given by
	\begin{equation*}
		\Ex[\big]{μ_n^ξ(f)^k}
		= \sum_{p=1}^k \sum_{\substack{Π∈\cQ_k\\\abs{Π}=p}} \int_{W_n^p} f\bigl(x_1n^{-\frac{1}{d}}\bigr)^{\abs{π_1}}\dotsm f\bigl(x_pn^{-\frac{1}{d}}\bigr)^{\abs{π_p}} m_{(\abs{π_1},…,\abs{π_p})}(\mult x;n) \dif \mult x.
	\end{equation*}
	Apply the moment-cumulant formula to this equation to obtain
	\begin{equation*}
		κ^{(k)}\bigl(μ_n^ξ(f)\bigr) = \sum_{p=1}^k \sum_{\substack{Π∈\cQ_k\\\abs{Π}=p}} \int_{W_n^p} f\bigl(x_1n^{-\frac{1}{d}}\bigr)^{\abs{π_1}} \dotsm f\bigl(x_p n^{-\frac{1}{d}}\bigr)^{\abs{π_p}} κ_{(\abs{π_1},…,\abs{π_p})}(\mult x;n) \dif \mult x.
	\end{equation*}
	Thus, \cref{lem:BoundFactorialCumulantDensity} implies the existence of a constant $C_1\ge1$ with
	\begin{align*}
		\MoveEqLeft \abs[\big]{κ^{(k)}\bigl(μ_n^ξ(f)\bigr)}
		\le \norm{f}_∞^k \sum_{p=1}^k \sum_{\substack{Π∈\cQ_k\\\abs{Π}=p}} \int_{W_n^p} \abs[\big]{κ_{(\abs{π_1},…,\abs{π_p})}(\mult x;n)} \dif \mult x\\
		&\le \norm{f}_∞^k C_1^k \sum_{p=1}^k \sum_{\substack{Π∈\cQ_k\\\abs{Π}=p}} \int_{W_n^p}\biggl( p! k!^{\max\set{γ_2,β}} e^{-c\max\set[\big]{\frac{\max_I \dist(\mult x_I,\mult x_{I^c})}{3},1}^{\hat{a}}}\\
		&\quad + p! k!^β C_1^N k^{N\max\set[\big]{\frac{1}{\hat{a}}+\frac{a}{d},\frac{1}{d}}} N!^b \Bigl(\max\set[\Big]{\frac{\max_{I}\dist(\mult x_I,\mult x_{I^c})}{3},1}\Bigr)^{-N\min\set[\big]{1,\frac{(1-a)\hat{a}}{d}}} \biggr) \dif \mult x.
	\end{align*}
	For the integral observe that, by translation invariance and \cref{cor:IntegrateTreeDist},
	\begin{align*}
		\MoveEqLeft \int_{W_n^p} \max\set[\Big]{\frac{\max_{I}\dist(\mult x_I,\mult x_{I^c})}{3},1}^{-N\min\set[\big]{1,\frac{(1-a)\hat{a}}{d}}} \dif \mult x\\
		&\le \Vol(W_n) * \int_{(ℝ^d)^{p-1}} \max\set[\Big]{\frac{\max_{I}\dist((0,\mult y_I),\mult y_{I^c})}{3},1}^{-N\min\set[\big]{1,\frac{(1-a)\hat{a}}{d}}} \dif \mult y\\
		&= 3^{d(p-1)} n \int_{(ℝ^d)^{p-1}} \max\set[\big]{\norm{\mult y}_{sig},1}^{-N\min\set[\big]{1,\frac{(1-a)\hat{a}}{d}}} \dif \mult y\\
		&\le 3^{d(p-1)} n C_2^p p!,
	\end{align*}
	where $N = \frac{dp}{\min\set{1,\frac{(1-a)\hat{a}}{d}}}$ was chosen and $C_2$ denotes the constant from \cref{cor:IntegrateTreeDist}. Similarly,
	\begin{align*}
		\MoveEqLeft \int_{W_n^p} e^{-c\max\set[\big]{\frac{\max_{I}\dist(\mult x_I,\mult x_{I^c})}{3},1}^{\hat{a}}} \dif \mult x\\
		&\le \Vol(W_n) * \int_{(ℝ^d)^{p-1}} e^{-c\max\set[\big]{\frac{\max_{I}\dist((0,\mult y_I),\mult y_{I^c})}{3},1}^{\hat{a}}} \dif \mult y\\
		&= 3^{d(p-1)} n \int_{(ℝ^d)^{p-1}} e^{-c\max\set[\big]{\frac{\max_{I}\dist((0,\mult y_I),\mult y_{I^c})}{3},1}^{\hat{a}}} \dif \mult y\\
		&\le 3^{d(p-1)} n C_2^p p!^{1+\frac{d}{\hat{a}}}.
	\end{align*}
	Combining both bounds with \cref{lem:SumSetPartition} enables us to conclude that there exists a constant $C_3\ge1$ such that
	\begin{align*}
		\MoveEqLeft \abs[\big]{κ^{(k)}\bigl(μ_n^ξ(f)\bigr)}\\
		&\le \norm{f}_∞^k C_1^k \sum_{p=1}^k \sum_{\substack{Π∈\cQ_k\\\abs{Π}=p}} \int_{W_n^p} \Bigl( p! k!^{\max\set{γ_2,β}} e^{-c\max\set[\big]{\frac{\max_I \dist(\mult x_I,\mult x_{I^c})}{3},1}^{\hat{a}}}\\
		&\quad + p! k!^β C_1^N k^{N\max\set{\frac{1}{\hat{a}}+\frac{a}{d},\frac{1}{d}}} N!^b \max\set[\Big]{\frac{\max_{I}\dist(\mult x_I,\mult x_{I^c})}{3},1}^{-N\min\set[\big]{1,\frac{(1-a)\hat{a}}{d}}} \Bigr) \dif \mult x\\
		&\le n\norm{f}_∞^k C_3^k \sum_{p=1}^k\sum_{\substack{Π∈\cQ_k\\\abs{Π}=p}} \biggl(  p!^2 k!^β k!^{\max\set[\big]{\frac{d}{\hat{a}}+a,\frac{d}{(1-a)\hat{a}}}} p!^{\max\set[\big]{db,\frac{bd^2}{(1-a)\hat{a}}}} + p!^2 k!^{\max\set{γ_2,β}} p!^\frac{d}{\hat{a}} \biggr)\\
		&\le n\norm{f}_∞^k (2C_3)^k \Bigl( k!^{2+β+ \max\set[\big]{\frac{d}{\hat{a}}+a,\frac{d}{(1-a)\hat{a}}} + \max\set[\big]{db,\frac{bd^2}{(1-a)\hat{a}}}} + k!^{2+\max\set{γ_2,β}+\frac{d}{\hat{a}}} \Bigr).
	\end{align*}
	This proves the result.
\end{proof}
\begin{remark}
	By interchanging the order of summation, i.e.\ first summing with respect to the set partition given in the proof of \cref{thm:BoundCumulants} and only afterwards with respect to the sums corresponding to the set partitions in \cref{lem:BoundFactorialCumulantDensity}, we obtain
	\begin{equation*}
		\abs[\big]{κ^{(k)}\bigl(μ_n^ξ(f)\bigr)} \le
		\begin{cases} n\norm{f}_∞^k C^k k!^{1 + \max\set{1,a+γ_2,a+β} + \frac{d}{(1-a)\hat{a}} + \frac{bd^2}{(1-a)\hat{a}}} & \text{if }\frac{(1-a)\hat{a}}{d}\le 1,\\
		n\norm{f}_∞^k C^k k!^{1 + \max\set{1,a+γ_2,a+β} + \frac{d}{\hat{a}} + a + bd} & \text{if }\frac{(1-a)\hat{a}}{d}\ge 1,
	\end{cases}
	\end{equation*}
	instead. This is a slight improvement and reflects the version obtained in \cite{ERS15MDPForStabilizingFunctionals}.
\end{remark}
\begin{remark}
	In the case of Poisson point process input we might choose $a=0$ and $\hat{a}=∞$ to obtain the bound on cumulants
	\begin{equation*}
		\abs[\big]{κ^{(k)}\bigl(μ_n^ξ(f)\bigr)}
		\le n\norm{f}_∞^k C^k k!^{2 + \max\set{γ_2,β} + bd}.
	\end{equation*}
	Besides the additional power growth condition reflected in the constant $γ_2$, we recover the result previously found in \cite{ERS15MDPForStabilizingFunctionals}. Actually, a closer look at the proof reveals that the case of Poisson input can also be dealt with, without assuming the power growth condition. This is due to the fact that the factorial moment expansion from \cref{lem:FMEexpansion} implies in this case that
	\begin{equation*}
		\abs[\big]{\tilde{m}_{\mult k}(\mult x;n)-\tilde{m}_{\mult k_I}(\mult x_I;n)\tilde{m}_{\mult k_{I^c}}(\mult x_{I^c};n)} = 0.
	\end{equation*}
\end{remark}
We are now able to conclude our asymptotic results from the bound on cumulants by using the following compilation of results presented in \cite{DE13MDPViaCumulants} and \cite{SS91LimitTheorems}. Recall that by $\Normaldist_{0,1}$ we denote a standard normal distributed random variable.
\begin{proposition}\label{thm:ConsequencesBoundCumulants}
	Consider a sequence $(X_n)_{n∈ℕ}$ of random variables with $\Ex{X_n}=0$ and $\Var{X_n}=1$ for all $n∈ℕ$. Assume there exists a constant $γ∈ℝ_+$ as well as a sequence $Δ_n∈\intoo{0,∞}$ such that
	\begin{equation*}
		\abs[\big]{κ^{(k)}(X_n)} \le Δ_n^{2-k}k!^{1+γ}
	\end{equation*}
	holds for all $k,n∈ℕ$. Then, the following statements are true:
	\begin{itemize}
		\item \emph{Berry-Esseen estimate:} There exists a constant $C>0$ such that
		\begin{equation*}
			\sup_{s∈ℝ} \abs[\big]{\Pr{X_n\le s} - \Pr{\Normaldist_{0,1}\le s}} \le CΔ_n^{-\frac{1}{1+2γ}}.
		\end{equation*}
		\item \emph{Concentration inequality:} For all $s∈ℝ_+$ and sufficiently large $n∈ℕ$ it holds
		\begin{equation*}
			\Pr[\big]{\abs{X_n} \ge s} \le 2 \exp\Bigl(-\frac{1}{4}\min\set[\Big]{\frac{s^2}{2^{1+γ}}, (sΔ_n)^\frac{1}{1+2γ}}\Bigr).
		\end{equation*}
		\item \emph{Moderate deviation principle:} For any sequence $(a_n)_{n∈ℕ}$ of real numbers with $\lim_{n→∞} a_n = ∞$ and $\lim_{n→∞} a_nΔ_n^{-\frac{1}{1+2γ}} = 0$, the sequence $(a_n^{-1}X_n)_{n∈ℕ}$ satisfies a moderate deviation principle on $ℝ$ with speed $a_n^2$ and Gaussian rate function $I(x)=\frac{x^2}{2}$.
	\end{itemize}
\end{proposition}
\begin{proof}[Proof of {\cref{thm:BerryEsseen,thm:Concentration,thm:MDP}}]
	We combine \cref{thm:BoundCumulants} and \cref{thm:ConsequencesBoundCumulants} to conclude: 
	Define
	\begin{equation*}
		γ=
		\begin{cases}
			1 + \max\set{γ_2,β} + \frac{d}{(1-a)\hat{a}} + \frac{bd^2}{(1-a)\hat{a}} & \text{if }\frac{(1-a)\hat{a}}{d} \le 1,\\
			1 + \max\set{γ_2,β} + \frac{d}{\hat{a}} + a + bd & \text{if }\frac{(1-a)\hat{a}}{d} \ge 1.
		\end{cases}
	\end{equation*}
	Then, the $k$-th cumulant is bounded by
	\begin{equation*}
		\abs[\bigg]{κ^{(k)}\biggl(\frac{μ_n^ξ(f)-\Ex[\big]{μ_n^ξ(f)}}{σ_n^ξ(f)}\biggr)}
		\le n \biggl(\frac{C_1\norm{f}_∞}{σ_n^ξ(f)}\biggr)^k k!^{1+γ}.
	\end{equation*}
	Since $\lim_{n→∞} \frac{σ_n^ξ(f)}{\sqrt{n}} = σ(ξ) \bigl(\int_{W_1} f(x)^2 \dif x\bigr)^\frac{1}{2} > 0$, there exists a constant $C_2>0$ such that $σ_n^ξ(f) \ge \sqrt{n}C_2^{-1} σ(ξ) (\int_{W_1} f(x)^2 \dif x)^\frac{1}{2}$. Define $C = \frac{σ(ξ)(\int_{W_1}f(x)^2 \dif x)^\frac{1}{2}}{C_1C_2\norm{f}_∞}$. Then,
	\begin{equation*}
		\abs[\bigg]{κ^{(k)}\biggl(\frac{μ_n^ξ(f)-\Ex[\big]{μ_n^ξ(f)}}{σ_n^ξ(f)}\biggr)}
		\le \sqrt{n}^{2-k} C^k k!^{1+γ}
	\end{equation*}
	 and we can thus apply \cref{thm:ConsequencesBoundCumulants} with $Δ_n = \frac{\sqrt{n}}{C\max\set{1,C^2}}$.
\end{proof}
\begin{proof}[Proof of {\cref{thm:SLLN}}]
	The concentration inequality from \cref{thm:Concentration} yields
	\begin{equation*}
		\Pr[\Big]{\abs[\big]{μ_n^ξ(f)-\Ex[\big]{μ_n^ξ(f)}} \ge s\sqrt{n}^{1+ε}}
		\le 2 \exp\Biggl(-\frac{1}{4}\min\set[\Bigg]{\frac{n^{1+ε}s^2}{σ_n^ξ(f)^2 2^{1+γ}}, C\biggl(\frac{n^{2+ε}s^2}{σ_n^ξ(f)^2}\biggr)^\frac{1}{2+4γ}}\Biggr).
	\end{equation*}
	Since $\lim_{n→∞} \frac{σ_n^ξ(f)}{\sqrt{n}} = σ(ξ)\bigl(\int_{W_1}f(x)^2\dif x\bigr)^\frac{1}{2}>0$, there exists a constant $C_1>0$ such that $σ_n^ξ(f) \le C_1 \sqrt{n}$. Hence,
	\begin{align*}
		\MoveEqLeft \Pr[\Big]{\abs[\big]{μ_n^ξ(f)-\Ex[\big]{μ_n^ξ(f)}} \ge s\sqrt{n}^{1+ε}}\\
		&\le 2 \exp\Biggl(-\frac{1}{4}\min\set[\Bigg]{\frac{n^εs^2}{2^{1+γ}C_1^2}, C\biggl(\frac{n^{1+ε}s^2}{C_1^2}\biggr)^\frac{1}{2+4γ}}\Biggr)\\
		&\le 2 \exp\biggl(-\frac{1}{4}\frac{n^εs^2}{2^{1+γ}C_1^2}\biggr) + 2 \exp\Biggl(C\biggr(\frac{n^{1+ε}s^2}{C_1^2}\biggl)^\frac{1}{2+4γ}\Biggr).
	\end{align*}
	Finally, notice that
	\begin{equation*}
		\sum_{n∈ℕ} 2\exp\biggl(-\frac{1}{4} \frac{n^εs^2}{C_1^22^{1+γ}}\biggr) < ∞
		\qquad\text{and}\qquad
		\sum_{n∈ℕ} 2\exp\Biggl(C\biggl(\frac{n^{1+ε}s^2}{C_1^2}\biggr)^\frac{1}{2+4γ}\Biggr) < ∞
	\end{equation*}
	and conclude via a Borel-Cantelli argument.
\end{proof}


\section*{Acknowledgement} \label{sec:Acknowledgement}

The author would like to thank Joseph Yukich for introducing him to this beautiful topic. He would also like to thank Peter Eichelsbacher for helpful discussions.


\bibliography{./SF_Literature.bib}

\end{document}